\documentclass[a4paper,12pt]{amsart}

\usepackage{fullpage}
\usepackage{latexsym,amsfonts,amssymb,amsmath,amsthm}

\usepackage[top=1in, bottom=1in, left=1in, right=1in]{geometry}
\usepackage{graphicx}
\usepackage{epstopdf}
\usepackage{verbatim}
\usepackage{enumerate}

\newtheorem{thm}{Theorem}[section]

\newtheorem{prop}[thm]{Proposition}
\newtheorem{lem}[thm]{Lemma}
\newtheorem{lemma}[thm]{Lemma}
\theoremstyle{remark}
\newtheorem{rem}{Remark}
\newcommand{\sumstar}{\sideset{}{^*}\sum}
\newcommand{\sumb}{\sideset{}{^\flat}\sum}
\newcommand{\sumd}{\sideset{}{^d}\sum}

\newcommand{\tRe}{\textup{Re }}

\newcommand{\bfrac}[2]{\left(\frac{#1}{#2}\right)}
\newcommand{\M}{\mathcal M}
\newcommand{\chiq}{\chi \textup{ (mod $q$)}}
\newcommand{\chir}{\chi \textup{ (mod } r)}

\newcommand{\cb}{\overline{\chi}}
\newcommand{\V}{V}
\newcommand{\W}{\mathcal W}
\newcommand{\D}{\mathcal D (\Psi, Q)}
\newcommand{\Df}{\mathcal D_f (\Psi, Q)}
\newcommand{\Sm}{\mathcal S (\Psi, Q)}
\newcommand{\Smf}{\mathcal S_f (\Psi, Q)}
\newcommand{\Smo}{\mathcal{S}_1 (\Psi, Q)}

\newcommand{\Mg}{\mathcal{MG}(\Psi, Q)}
\newcommand{\Eg}{\mathcal{EG}(\Psi, Q)}
\newcommand{\G}{\mathcal G (\Psi, Q)}
\newcommand{\Gf}{\mathcal G_f (\Psi, Q)}
\newcommand{\Hc}{\mathcal H}
\newcommand{\B}{\mathcal B}
\newcommand{\F}{\mathcal F}
\newcommand{\R}{{\mathrm{Re}}}

\newcommand{\h}{\frac{1}{2}} 

\newcommand{\lon}{(\log Q)^{\alpha}}
\newcommand{\lotwo}{(\log Q)^{2\alpha}}
\newcommand{\phib}{\phi^\flat(q)}
\newcommand{\dt}{\widetilde{\Delta}}
\newcommand{\Wt}{\widetilde{\mathcal{W}}}
\newcommand{\psit}{\widetilde{\Psi}}

\begin{document}
\title{The eighth moment of Dirichlet $L$-functions} 
\author{Vorrapan Chandee}
\address{Department of Mathematics \\ Burapha University \\ 169 Long-Hard Bangsaen Road, Saensook Municipality, Muang District, Chonburi, Thailand 20131
}
\email{vorrapan@buu.ac.th}

\author{Xiannan Li}
\address{University of Oxford \\
Andrew Wiles Building\\
Radcliffe Observatory Quarter\\
Woodstock Road,  Oxford\\
UK\\
OX2 6GG}
\email{lix1@maths.ox.ac.uk}



\begin{abstract} 
We prove an asymptotic for the eighth moment of Dirichlet $L$-functions averaged over primitive characters $\chi$ modulo $q$, over all moduli $q\leq Q$ and with a short average on the critical line, conditionally on GRH.  We derive the analogous result for the fourth moment of Dirichlet twists of $GL(2)$ L-functions.  Our results match the moment conjectures in the literature; in particular, the constant 24024 appears as a factor in the leading order term of the eighth moment.
\end{abstract}

\maketitle
\section{Introduction} 
There has been substantial and sustained research into moments of $L$-functions on the critical line.  Much of this interest is generated by the presence of numerous applications, but moments are also studied for their own intrinsic interest.  As well, understanding moments of $L$-functions involves developing tools which better elucidates the nature of the approximate orthogonality of the family of $L$-functions under consideration.   

The first moments studied were those of the Riemann zeta function, which are averages of the form
$$I_k(T) := \int_0^T |\zeta(\tfrac{1}{2} + it)|^{2k}dt.
$$  Here, asymptotic formulae were proven for $k = 1$ by Hardy and Littlewood and for $k=2$ by Ingham (see \cite{Ti} VII).  Despite extensive further work, including various refinements of the result of Ingham, no such result is available for any other values of $k$.  

A well known folklore conjecture states that  $I_k(T) \sim c_k T(\log T)^{k^2}$ for constants $c_k$ depending on $k$. The values of $c_k$ were mysterious for general $k$ until the work of Keating and Snaith \cite{KS} which related these moments to circular unitary ensembles and provided precise conjectures for $c_k$.  The choice of group is consistent with the Katz-Sarnak philosophy \cite{KaSa}, which indicates that the symmetry group associated to this family should be unitary.  Based on heuristics for shifted divisor sums, Conrey and Ghosh derived a conjecture in the case $k=3$ \cite{CGh} and Conrey and Gonek derived a conjecture in the case $k=4$ \cite{CGo}.  Further conjectures including lower order terms, and for other symmetry groups are available from the work of Conrey, Farmer, Keating, Rubinstein and Snaith \cite{CFKRS} as well as from the work of Diaconu, Goldfeld and Hoffstein \cite{DGH}.

In support of these conjectures, lower bounds of the the right order of magnitude are available due to Rudnick and Soundararajan \cite{Lower}, while good upper bounds are available conditionally on RH, due to Soundararajan \cite{Sound}, of the form $I_k(T) \ll T(\log T)^{k^2+\epsilon}$ .

The situation for other families of $L$-functions is very similar; asymptotics are only available for small values of $k$.  While the asymptotic for the sixth moment of $\zeta(1/2+it)$ remains out of reach, Conrey, Iwaniec and Soundararajan \cite{CIS} have recently derived an asymptotic formula for the sixth moment of Dirichlet $L$-functions with a power saving error term.  Instead of fixing the modulus $q$ and only averaging over characters $\chiq$, they also average over the modulus $q \leq Q$, which gives them a larger family of size $Q^2$, and further they include a short average on the critical line.  Since the family they consider is also unitary, the asymptotic is very similar to the conjectured asymptotic for the sixth moment of $\zeta(1/2+it)$.

To be more precise, let $\chiq$ be a primitive even Dirichlet character, and let (for $\tRe s > 1$),
$$ L(s, \chi) = \sum_{n=1}^\infty \frac{\chi(n)}{n^s} = \prod_p \left( 1 - \frac{\chi(p)}{p^s}\right)^{-1} $$ 
be the $L$-function associated to it.  Then the completed $L$-function $\Lambda(s, \chi)$ satisfies
$$ \Lambda\big( \tfrac{1}{2} + s, \chi \big)  = \left(\frac{q}{\pi}\right)^{s/2} \Gamma \left(\tfrac{1}{4} + \tfrac{s}{2}\right) L\big( \tfrac{1}{2} + s, \chi\big) = \epsilon_\chi \Lambda\big( \tfrac{1}{2} - s, \overline{\chi}\big),$$
where $|\epsilon_\chi| = 1$.  

Let $\sumb_{\chiq}$ denote a sum over primitive even Dirichlet characters with modulus $q$, and $\phib$ denote the number of primitive even Dirichlet characters with modulus $q$.  Then Corollary 1 in the work of Conrey, Iwaniec, and Soundararajan \cite{CIS} states
\begin{align*}
 \sum_{q\leq Q} \;\; &\sumb_{\chiq} \int_{-\infty}^{\infty} \left| \Lambda\big( \tfrac{1}{2} + iy, \chi \big)\right|^6 \> dy \\
&\sim 42 a_3\sum_{q\leq Q}  \prod_{p|q} \frac{\left( 1 - \frac{1}{p}\right)^5}{\left( 1 + \frac{4}{p} + \frac{1}{p^2}\right)} \phib \frac{(\log q)^{9}}{9!} \int_{-\infty}^{\infty} \left| \Gamma \left( \frac{1/2 + iy}{2}\right)\right|^6 \> dy,
\end{align*}where
$$a_3 = \prod_p \left(1-\frac{1}{p^4}\right)\left(1+\frac{4}{p} + \frac{1}{p^2}\right).
$$This is consistent with the conjecture in \cite{CFKRS}, and fits with the analogous conjecture for $\zeta(1/2+it)$ in \cite{CGo}.  The authors of \cite{CIS} later state a more precise technical result which gives the asymptotic for the sixth moment including shifts with a power saving error term.  The average over $y$ is fairly short due to the rapid decay of the $\Gamma$ function along vertical lines, but deriving an analogous result without the average over $y$ remains an open problem which involves understanding certain unbalanced sums.  Here, the restriction to even characters is a technical convenience, and the analogous result may be derived for odd characters using the same method.


In this paper, we shall derive asymptotics for the eighth moment of Dirichlet $L$-functions and for the fourth moment of Dirichlet twists of certain $GL(2)$ $L$-functions, conditionally on GRH. 

From \cite{CFKRS}, we may derive the conjecture that as $q \rightarrow \infty$ with $q \neq 2$ (mod 4),
$$ \frac{1}{\phib} \sumb_{\chiq} \big|L\big(\tfrac{1}{2}, \chi \big)\big|^8 \sim 24024 \ a_4 \prod_{p | q} \frac{\left( 1 - \frac{1}{p}\right)^7}{\left( 1 + \frac{9}{p} + \frac{9}{p^2} + \frac{1}{p^3}\right)} \frac{(\log q)^{16}}{16!},$$
where 
\begin{equation} \label{eqn:arithfactor}
a_4 = \prod_{p} \left( 1 - \frac{1}{p}\right)^9 \left(1 + \frac{9}{p} + \frac{9}{p^2} + \frac{1}{p^3} \right).
\end{equation}
Towards this conjecture, we shall prove that
\begin{align}
\label{eqnCor8}
 \sum_{q \leq Q} \ \sumb_{\chiq} &\int_{-\infty}^{\infty} \left| \Lambda\big( \tfrac{1}{2} + iy, \chi \big)\right|^8  \> dy \notag \\
&\sim 24024\ a_4\sum_{q \leq Q} \prod_{p|q} \frac{\left( 1 - \frac{1}{p}\right)^7}{\left( 1 + \frac{9}{p} + \frac{9}{p^2} + \frac{1}{p^3}\right)} \phib \frac{(\log q)^{16}}{16!} \int_{-\infty}^{\infty} \left| \Gamma \left( \frac{1/2 + iy}{2}\right)\right|^8 \> dy 
\end{align}conditionally on GRH.

Our result (\ref{eqnCor8}) follows immediately from the following theorem.
\begin{thm}\label{thm:eightmoment} Assume GRH. Let $\Psi$ be a smooth function compactly supported in [1, 2]. Then, we have
\begin{align*}
 \sum_{q } \Psi\left(\frac{q}{Q}\right) &\sumb_{\chiq} \int_{-\infty}^{\infty} \left| \Lambda\big( \tfrac{1}{2} + iy, \chi \big)\right|^8 \> dy \\
&= 24024\ a_4\sum_{q}  \Psi\left(\frac{q}{Q}\right) \prod_{p|q} \frac{\left( 1 - \frac{1}{p}\right)^7}{\left( 1 + \frac{9}{p} + \frac{9}{p^2} + \frac{1}{p^3}\right)} \phib \frac{(\log q)^{16}}{16!} \int_{-\infty}^{\infty} \left| \Gamma \left( \frac{1/2 + iy}{2}\right)\right|^8 \> dy \\
& \hskip 2in + O(Q^2(\log Q)^{15+ \epsilon}). 
\end{align*}
\end{thm}

\begin{rem}
The main term in the theorem is of the order $Q^2 (\log Q)^{16}$.
\end{rem}
Our method also allows us to compute the fourth moment of Dirichlet twists of a $GL(2)$ automorphic $L$-function.  To be precise, let $f$ be a holomorphic modular form of weight $k$ and full level\footnote{Our methods apply to other $GL(2)$ forms; we restrict our attention to this case to ease notation.}. We write
$$L(s, f) = \sum_{n} \frac{a(n)}{n^s} = \prod_p \left(1 - \frac{\alpha_p}{p^s} \right)^{-1}\left(1 - \frac{\beta_p}{p^s} \right)^{-1},
$$normalized so that the critical line is at $\tRe s = 1/2$.  Then the twisted $L$-function $L(s, f\times \chi)$ has Dirichlet series
$$L(s, f\times \chi) = \sum_n \frac{a(n)\chi(n)}{n^s},
$$and satisfies the functional equation
\begin{equation} \label{eqn:fncEqnoftwistedL}
\Lambda\left(\tfrac{1}{2} + s, f\times \chi\right) :=    \bfrac{q}{2\pi}^s \Gamma\left(s+ \tfrac{k}{2}\right) L\left(s + \tfrac{1}{2}, f\times \chi\right) = \omega_\chi \Lambda\left(\tfrac{1}{2} - s, f\times \cb\right), 
\end{equation}
where $\omega_\chi = i^k \frac{\tau(\chi)^2}{q}$.  Here, $\tau(\chi)$ is the Gauss sum and so $|\omega_\chi| = 1.$ Now let 
$$ f_2 := \prod_p \left(1 - \frac{a(p)^2 - 2}{p} + \frac{1}{p^2}\right)^{-3} \left(1 + \frac{a(p)^2 + 2}{p} - \frac{4a(p)^2 - 2}{p^2} + \frac{a(p)^2 + 2 }{p^3} + \frac{1}{p^4}  \right),$$
and $B_p(f, 1/2)$ be defined by
\begin{equation}\label{def:Bpf}
 \left(1 - \frac{1}{p} \right)^{-4} \left(1 - \frac{1}{p^2} \right)\left(1 - \frac{a(p)^2 - 2}{p} + \frac{1}{p^2}\right)^{-3} \left(1 + \frac{a(p)^2 + 2}{p} - \frac{4a(p)^2 - 2}{p^2} + \frac{a(p)^2 + 2 }{p^3} + \frac{1}{p^4}  \right).
\end{equation}

Then we have the following
\begin{thm} \label{thm:twisted4moment}
Assume GRH. Let $\Psi$ be a smooth function compactly supported in [1, 2]. Then, we have
\begin{align*}
 \sum_{q } \Psi\left(\frac{q}{Q}\right) &\sumstar_{\chiq} \int_{-\infty}^{\infty} \left| \Lambda\big( \tfrac{1}{2} + iy, f \times \chi \big)\right|^4 \> dy \\
&= \frac{1}{2\pi^2}\ f_2\sum_{q}  \Psi\left(\frac{q}{Q}\right)  \phi^*(q) (2\log q)^{4} \prod_{p|q} \frac{1}{B_p(f,1/2)} \int_{-\infty}^{\infty} \left| \Gamma \left( \frac{k/2 + iy}{2}\right)\right|^4 \> dy \\
& \hskip 2in + O(Q^2(\log Q)^{3+ \epsilon}). 
\end{align*}
\end{thm}

\begin{rem}
The main term in the theorem is of the order $Q^2 (\log Q)^4$.
\end{rem}

This is consistent with the conjectures for the fourth moment of a unitary family of $L$-functions, as in \cite{CFKRS}.  The proof of Theorem \ref{thm:twisted4moment} is similar to that of Theorem \ref{thm:eightmoment}.  In the sequel, we will focus on proving Theorem \ref{thm:eightmoment} and indicate how to modify the proof for Theorem \ref{thm:twisted4moment} in Section \ref{sec:twistedfourthmoment}.

\subsection{A sketch of the proof}
As the proof contains a number of technical details, we first provide a sketch which will hopefully indicate the main features.  Roughly speaking, after applying the approximate functional equation, we need to understand sums of the form
$$Q\sum_{q} \Psi\bfrac{q}{Q} \sum_{m \leq Q^2} \sum_{\substack{ n \leq Q^2 \\ m\equiv n \textup{(mod $q$)}}} \frac{\tau_4(m)\tau_4(n)}{\sqrt{mn}},
$$
where $\tau_4(n) = \sum_{n = n_1n_2n_2n_4} 1.$ The diagonal contribution $m=n$ may be extracted and is fairly easy to understand.  In the case of the sixth moment (and in their study of the asymptotic large sieve), Conrey, Iwaniec and Soundararajan \cite{CIS} use the complementary divisor trick, to write $m - n = hq$ and replace the congruence condition modulo $q$ with a congruence condition modulo $h$.  In the case of the sixth moment, the critical range is $m\asymp n \asymp Q^{3/2}$ and since $h$ is of size $|m-n|/Q$, this results in a massive reduction in conductor in the work \cite{CIS}.  However, in our case, the critical range is $m\asymp n \asymp Q^2$, so that applying the complementary divisor trick immediately is futile.  Thus before switching to the complementary divisor, we truncate our sums in $m$ and $n$ to $Q^{2-\epsilon}$, directly using the multiplicative large sieve (see Section \ref{sec:truncation}).  

After switching to the complementary divisor, we express the congruence condition modulo $h$ using characters modulo $h$, so that roughly, we want to study
$$Q\sum_{h} \frac{1}{\phi(h)}\sum_{\chi \textup{(mod $h$)}}\sum_{m \leq Q^{2-\epsilon}} \sum_{\substack{ n \leq Q^{2-\epsilon}}} \frac{\tau_4(m)\tau_4(n)\chi(m) \overline{\chi(n)}}{\sqrt{mn}}\Psi\bfrac{|m-n|}{hQ}\frac{|m-n|}{hQ}V(m, n) .
$$In the above, the principal characters give a main term contribution while the rest should be absorbed into the error term.  If the essentially smooth factor $\Psi\bfrac{|m-n|}{hQ}\frac{|m-n|}{hQ}V(m, n)$ did not exist, one may separate $m$ and $n$ and apply the large sieve again to give a satisfactory bound.  Unfortunately, separating the variables $m$ and $n$ using Mellin transforms (see \S \ref{sec:Mellin}) comes at a high cost.  To be more precise, note that the $k$th derivative of $\Psi\bfrac{|m-n|}{hQ}$ with respect to $m$ is of size $\bfrac{1}{hQ}^k$.  In order for the Mellin transform to decay rapidly we would like this to be $\bfrac{1}{m}^k$, which is only true when the complementary divisor $h$ is of size around $Q^{1-\epsilon}$ (the reader should think of $m$ being around size $Q^{2-\epsilon}$ as this is the critical range).  It is impossible to ignore the smaller values of $h$; for instance, they genuinely contribute to the main term.  Technically, GRH is used to bound an eighth moment of Dirichlet $L$-functions in the $t$-aspect, which appears since the Mellin transform does not decay rapidly.  Less technically, GRH guarantees significant cancellation in the sums over $m$ and $n$ even in short intervals - visibly, for fixed $n$, the appearance of $\Psi\bfrac{|m-n|}{hQ}$ restricts $m$ to an interval much shorter than $Q^{2-\epsilon}$ when $h$ is small.  

In this description, we have neglected to mention the effect of the inclusion-exclusion within the orthogonality over primitive characters, a number of coprimality conditions, and a rather lengthy list of technical manipulations.

\section{Notation and preliminary lemmas}
In this section, we introduce some notation and record some standard results.  Define
$$ \Lambda(s, \chi; t) := \Lambda^4(s + it, \chi) \Lambda^4(s - it, \overline\chi),$$
and
$$G(s, t) := \Gamma^4\left(\tfrac{s}{2} + \tfrac{it}{2}\right)\Gamma^4\left(\tfrac{s}{2}-\tfrac{it}{2}\right).$$
We then have
$$ \Lambda\left(\tfrac{1}{2}, \chi ; t\right) = G\left(\tfrac 12, t\right) L^4 \left( \tfrac{1}{2} + it, \chi\right)L^4\left(\tfrac{1}{2} - it,  \cb \right).$$
Therefore for $ \tRe (s) > 1,$ we write
$$  L^4 \left( s + it, \chi\right)L^4\left(s - it,  \cb \right) = \sum_{m, n = 1}^\infty \frac{\tau_4 (m) \tau_4 (n)}{m^s n^s} \chi(m) \cb(n) \left( \frac{n}{m}\right)^{it}.$$   

Moreover, define
\begin{equation}
\label{eqnW}
W(x, t) := \frac{1}{2\pi i} \int_{(1)} G(1/2+s, t) x^{-s} \frac{ds}{s},
\end{equation}

$$P(\chi , t) := \sum_{m, n=1}^\infty \frac{\tau_4(m) \tau_4(n)\chi(m)\cb(n)}{\sqrt{mn}}W\left(\frac{mn\pi^4}{q^4}, t\right),
$$
\begin{equation}
\label{eqnV}
 V(\xi, \eta; \mu) = \int_{-\infty}^{\infty} \left(\frac{\eta}{\xi}\right)^{it} W\left( \frac{\xi\eta \pi^4}{\mu^4}, t\right) \> dt,
\end{equation}
and 
$$ \Lambda_1(\chi) = \sum_{m, n = 1}^{\infty} \frac{\tau_4(m) \tau_4(n)}{\sqrt{mn}} \chi(m) \overline{\chi}(n) V(m, n ;q).$$
Our first Lemma is an approximate functional equation.
\begin{lemma} \label{prop:fnceqnLambda} With notation as above, we have
\begin{equation}
\label{eqn:appfe}
\Lambda(1/2, \chi; t) = 2P(\chi, t),
\end{equation}
and
\begin{equation}
\label{eqn:appfet}
\int_{-\infty}^{\infty} \Lambda(1/2, \chi ; t) \> dt = 2\Lambda_1(\chi).
\end{equation}

\end{lemma}

\begin{proof} 
We begin by writing
$$P(\chi, t) = \frac{1}{2\pi i} \int_{(1)} \Lambda(1/2+s, \chi; t) \frac{ds}{s}.
$$Shifting contours of integration to $\tRe s = -1$, we pass a simple pole at $s=0$ with residue $\Lambda(1/2, \chi; t)$.  Upon applying the functional equation, the remaining integral is
$$
\frac{1}{2\pi i} \int_{(-1)} \Lambda(1/2-s, \chi; t) \frac{ds}{s}
=-\frac{1}{2\pi i} \int_{(1)} \Lambda(1/2+s, \chi; t) \frac{ds}{s}
=-P(\chi, t).
$$This proves (\ref{eqn:appfe}).  Integrating both sides of (\ref{eqn:appfe}) gives (\ref{eqn:appfet}).
\end{proof}

The integration of $\Lambda(1/2, \chi ; t)$ in $t$ is needed to restrict the range of $m, n$. From the following lemma, we see that the main term of $\Lambda_1(\chi)$ comes from when $m, n$ are both at most $q^{2 + \epsilon}.$
\begin{lemma} \label{lem:weightWV} The weight function $W(x, t)$ defined in (\ref{eqnW}) is a smooth function of $x \in (0, \infty)$ and moreover, for any $x > 1$ and any non-negative integer $\nu$, 
$$ W^{(\nu)}(x, t) \ll_{\nu} \exp(-c_0 x^{1/4})$$
for some absolute constant $c_0 > 0.$  This in term gives us that the function $V$ defined in (\ref{eqnV}) satisfies
$$ V(\xi, \eta; \mu) \ll \exp \left(-c_1 \left(\frac{\max(\xi, \eta)^2}{\mu^4} \right) \right)^{1/4}.$$ 
\end{lemma}
\begin{proof} This lemma is the same as Lemma 1 in \cite{CIS}.  
\end{proof}

Next, we introduce some notation which will be used when calculating the arithmetic factor $a_4$ in (\ref{eqn:arithfactor}). 
Let $$ \B_p(s) = \sum_{r = 0}^{\infty} \frac{\tau_4^2(p^r)}{p^{2rs}}, \ \ \ \ \  \ \ \ \  \B_q(s) = \prod_{p|q} B_p(s),$$
and $$ A_p(s) =  \left(1 - \frac{1}{p^{2s}} \right)^{16} B_p(s) , \ \ \ \ \ \ \ \ A(s) = \prod_p A_p(s).$$ 
\begin{lemma} \label{lem:eulerprod} With notation as above, we have that for $\tRe (s) > 1/2$,
$$B_p(s) = \left( 1 - \frac{1}{p^{2s}}\right)^{-7}\left(1 + \frac{9}{p^{2s}} + \frac{9}{p^{4s}} + \frac{1}{p^{6s}}\right),$$
and
$$ \sum_{\substack{n = 1 \\ (n,q) = 1}}^{\infty} \frac{\tau_4^2(n)}{n^{2s}} =  \frac{\zeta^{16}(2s) A(s)}{\mathcal B_q(s)} ,$$
where $A(s)$ converges absolutely when $\tRe (s) > 1/4.$

\end{lemma}
\begin{proof} This is a standard proof which involves writing out both sides in terms of Euler products.  See Section 2 of \cite{CFKRS} for similar proofs.
\end{proof}

The next lemma is required to compute the main term of the fourth moment of Dirichlet twists of a $GL(2)$ automorphic $L$-function. We define 
 $$ \B_p(f, s) = \sum_{r = 0}^{\infty} \frac{\sigma^2_f(p^r)}{p^{2rs}}, \ \ \ \ \  \ \ \ \  \B_q(f, s) = \prod_{p|q} B_p(f,s),$$
and $$ A(f, s) =  \prod_p \left(1 - \frac{1}{p^{2s}}\right)^{4} B_p(f,s).$$

\begin{lemma} \label{lem:eulerprodfortwistedfourth} With notation as above, we have for $\tRe (s) > 1/2,$ 
\begin{equation} \label{eqn:sumtozeta}
\sum_{\substack{n = 1 \\ (n,q) = 1}}^{\infty} \frac{\sigma^2_f(n)}{n^{2s}} =  \frac{\zeta^{4}(2s) A(f,s)}{\mathcal B_q(f,s)} ,
\end{equation}
where $B_p(f, 1/2)$ is as in (\ref{def:Bpf}) and 
$$ A_p(f, 1/2) = \frac{6}{\pi^2}\prod_p \left(1 - \frac{a(p)^2 - 2}{p} + \frac{1}{p^2}\right)^{-3} \left(1 + \frac{a(p)^2 + 2}{p} - \frac{4a(p)^2 + 2}{p^2} + \frac{a(p)^2 + 2}{p^3} + \frac{1}{p^4}  \right),$$
converges absolutely when $\tRe (s) > 1/4.$
\end{lemma}
\begin{proof} Let $\oint$ denote the contour integral around the unit circle.  The proof of (\ref{eqn:sumtozeta}) can be found in \cite{CFKRS}, where $B_p(f,1/2)$ is given by
\begin{align*}
B_p(f, 1/2) &= \int_0^1 \left( 1 - \frac{\alpha_p e(\theta)}{p^{1/2}}\right)^{-2}\left( 1 - \frac{\beta_p e(\theta)}{p^{1/2}}\right)^{-2}\left( 1 - \frac{\alpha_p e(-\theta)}{p^{1/2}}\right)^{-2}\left( 1 - \frac{\beta_p e(-\theta)}{p^{1/2}}\right)^{-2} \> d\theta \\ 
&= \frac{p^2}{2\pi i} \oint z^3 \frac{\left( z - \frac{p^{1/2}}{\alpha_p}\right)^{-2}\left( z - \frac{p^{1/2}}{\beta_p}\right)^{-2}}{\left( z - \frac{\alpha_p }{p^{1/2}}\right)^{2}\left( z - \frac{\beta_p }{p^{1/2}}\right)^{2}} \> dz \\
&= p^2 \left({\rm Res}_{z = \tfrac{\alpha_p}{p^{1/2}}} + {\rm Res}_{z = \tfrac{\beta_p}{p^{1/2}}}\right) z^3 \frac{\left( z - \frac{p^{1/2}}{\alpha_p}\right)^{-2}\left( z - \frac{p^{1/2}}{\beta_p}\right)^{-2}}{\left( z - \frac{\alpha_p }{p^{1/2}}\right)^{2}\left( z - \frac{\beta_p }{p^{1/2}}\right)^{2}}.
\end{align*}
Using that $\alpha_p \beta_p = 1,$ the sum of the residues is
\begin{align*}
\frac{1}{p^6}\left(1 + \frac{1}{p}\right)\left( 1 - \frac{1}{p}\right)^{-3}&\left(1 - \frac{\alpha_p^2}{p}\right)^{-3}\left( 1 - \frac{\beta_p^2}{p}\right)^{-3} \times \\
&\times\left[(1 + 2p + 2p^2 + 2p^3 + p^4) + (\alpha_p + \beta_p)^2(p - 4p^2 + p^3)\right],
\end{align*}where gives the desired form for $B_p(f, 1/2)$ upon using that $\alpha_p + \beta_p = a(p)$.
\end{proof}

Next, we record the classical multiplicative large sieve as well as its direct consequence for the eighth moment which are used in Section \ref{sec:truncation} and \ref{sec:calS}. The proofs can be found in \cite{IK}. 
\begin{lemma}\label{prop:largesieve} 
For any complex number $a_n$ with $M < n < M + N$, where $N$ is a positive integer, we have
$$ \sum_{q \leq Q} \frac{q}{\phi (q)} \sumstar_{ \chiq} \left| \sum_{M < n < M+ N} a_n \chi(n)\right|^2 \leq (Q^2 + N) \sum_{M < n < M + N} |a_n|^2.$$
\end{lemma}

\begin{lemma} \label{prop:LfunctionofLargesieve} For any $t \in \mathbb R$, we have
$$  \sum_{q \leq Q} \  \sumstar_{ \chiq} \left|L\left(\tfrac 12 + it, \chi \right)\right|^8 \ll Q^2 (t^2 + 1)(\log Q(|t| + 2))^{17},$$
where the implied constant is absolute. 

\end{lemma}

We also need orthogonality relations for characters. 
\begin{lemma} \label{lem:orthogonal} If $m, n$ are integers with $(mn, q) = 1$ then
$$ \sumstar_{\chiq} \chi(m) \cb(n) =  \sum_{\substack{q = dr \\ r | (m - n)}} \mu(d) \phi(r),$$
and
$$ \sumb_{\chiq} \chi(m) \cb(n) = \frac{1}{2} \sum_{\substack{q = dr \\ r | (m \pm n)}} \mu(d) \phi(r).$$
\end{lemma}
\begin{proof} 
The first claim follows from the orthogonality of all characters and Mobius inversion, while the second claim follows from the first by detecting even characters with $\frac{1+\chi(-1)}{2}$.
\end{proof}

\begin{lemma} \label{lem:sumcharintS} Let $y$, $t$ and $S > 0$ be real numbers and $q$ and $k$ be natural numbers with $ y^k \leq \frac{\sqrt{\phi(q)(S+1)} }{\log (q(S + 1))}$.  For any complex numbers $a(p)$, we have that
$$ \sum_{\chi \ {\rm mod} \ q} \int_{S}^{2S} \left| \sum_{ p \leq y} \frac{a(p)\chi(p)}{p^{1/2 + it}}\right|^{2k} \> dt \ll \phi(q) (S + 1) k! \left(\sum_{p \leq y} \frac{|a(p)|^2}{p}\right)^k,$$
where the implied constant is absolute.
\end{lemma}
\begin{proof} 
Write $$\left(\sum_{p \leq y} \frac{a(p)\chi(p)}{p^{1/2 + it}}\right)^k = \sum_{n \leq y^k} \frac{a_{k, y}(n)\chi(n)}{n^{1/2 + it}}, $$
where $a_{k,y}(n) = \sum_{\substack{p_1...p_k = n \\ p_i \leq y} }a(p_1)...a(p_{k}).$
Then by orthogonality of Dirichlet characters
\begin{align*}
 &\sum_{\chi \ {\rm mod} \ q} \int_{S}^{2S}\left| \sum_{ p \leq y} \frac{a(p)\chi(p)}{p^{1/2 + it}}\right|^{2k} \> dt =  \phi(q) \sum_{\substack{m,n\leq y^k \\ (mn, q) = 1 \\ q | m - n}} \frac{a_{k,y}(m)\overline{a_{k,y}(n)}}{\sqrt{mn}} \int_S^{2S} \left( \frac{n}{m}\right)^{it}\> dt \\
&= \phi(q) S\sum_{\substack{n \leq y^k \\ (n,q) = 1}} \frac{|a_{k,y}(n)|^2}{n} + O\left(\phi(q) \sum_{\substack{m,n\leq y^k \\ m \neq n, q|m-n}} \frac{|a_{k,y}(m)\overline{a_{k,y}(n)}|}{\sqrt{mn} |\log (m/n)|} \right).
\end{align*}
The sum in the diagonal term is bounded by
\begin{align*}
\leq \sum_{\substack{n \leq y^k}} \frac{|a_{k,y}(n)|^2}{n}
&= \sum_{p_1,..., p_r \leq y} \sum_{\substack{\alpha_1, ..., \alpha_r \geq 1 \\ \sum_{\alpha_i = k}} } {k \choose \alpha_1,...,\alpha_k}^2 \frac{|a_{p_{1}}|^{2\alpha_1}...|a_{p_r}|^{2\alpha_r}}{p_1^{\alpha_1}...p_r^{\alpha_r}} \\
&\leq k! \sum_{p_1,..., p_r \leq y} \sum_{\substack{\alpha_1, ..., \alpha_r \geq 1 \\ \sum_{\alpha_i = k}} } {k \choose \alpha_1,...,\alpha_k} \frac{|a_{p_{1}}|^{2\alpha_1}...|a_{p_r}|^{2\alpha_r}}{p_1^{\alpha_1}...p_r^{\alpha_r}} \\
&= k! \left(\sum_{p \leq y} \frac{|a(p)|^2}{p}\right)^k.
\end{align*}

Now we return to the error term. Since $y^{k} \leq \frac{\sqrt{\phi(q)(S + 1)} }{\log (q(S+1))}$, either $y^k \leq \frac{\phi(q)}{\log (q(S + 1))}$ or $y^k \leq \frac{S + 1 }{\log (q(S+1))}.$

If $y^k \leq \frac{\phi(q)}{\log (q(S+1))}$, then $q|m - n$ if and only if $m =n$ so there are no additional terms. If $y^k \leq \frac{S+1 }{\log (q(S+1))},$ then the error term is bounded by
\begin{align*}
\ll \phi(q) \sum_{\substack{m\leq y^k}} \frac{|a_{k,y}(m)|^2}{m} \sum_{\substack{n \neq m \\ n \leq y^k}} \frac{1}{|\log (m/n)|}  &\ll \phi(q) y^k \log (y^k) \sum_{\substack{m\leq y^k}} \frac{|a_{k,y}(m)|^2}{m} \\
&\ll \phi(q)(S+1) k! \left(\sum_{p \leq y} \frac{|a(p)|^2}{p}\right)^k.
\end{align*}

\end{proof}

Finally, the following propositions estimate the moments of Dirichlet $L$-functions and Dirichlet twists of a $GL(2)$ automorphic $L$-function. They are crucial tools in bounding off-diagonal terms in Section \ref{sec:errorEg}.

\begin{prop} \label{thm:8momentIntandSumoverq} Assume GRH.  For $q \geq 3$, $S > 0$  and for any positive real number $k$ and any $\epsilon > 0$ we have
\begin{align*}
& \sum_{\chi \ (\rm{mod} \ q)} \int_{S \leq |s| < 2S}\left|L\left( \frac{1}{2} + s, \chi \right) \right|^{2k} ds\ll_{k, \epsilon} \phi(q) (S + 1) (\log q(S + 1))^{k^2 + \epsilon},
\end{align*}where the integration over $s$ is taken to be on a vertical line with $0\leq \tRe(s) \leq \frac{1}{\log Q}$.

\end{prop}
\begin{proof}
To prove the Theorem, it is enough to show that 
\begin{equation} \label{eqn:sumoverprimMoment}
 \sumstar_{\chi \ (\rm{mod} \ q)} \int_{S \leq |s| < 2S}\left|L\left( \frac{1}{2} + s, \chi \right) \right|^{2k} ds\ll_{k, \epsilon} \phi(q) (S+1) (\log q(S+1))^{k^2 + \epsilon}
\end{equation}
 since we can deduce the result for sum over all characters as the following.
\begin{align*}
&\sum_{\chi \ (\rm{mod} \ q)} \int_{S\leq |s| \leq 2S} \left|L\left( \frac{1}{2} + s, \chi \right) \right|^{2k} \> ds\\
&\leq \sum_{q_1 | q} \ \sumstar_{\chi_1 \ (\rm{mod} \ q_1)} \left(\int_{S\leq |s| \leq 2S} \left|L\left( \frac{1}{2} + s, \chi_1 \right) \right|^{2k}  \>ds \right)\prod_{\substack{p|q \\ p \nmid q_1}} \left(1 + \frac{1}{p^{1/2}}\right)^{2k}\\
&\ll_{k, \epsilon}  (S+1)(\log qS)^{k^2 + \epsilon} \sum_{q_1|q}\phi(q_1) \prod_{\substack{p|q \\ p \nmid q_1}} \left(1 + \frac{1}{p^{1/2}}\right)^{2k}. \\
&\ll_{k, \epsilon} q(S+1)(\log q)^{k^2 +\epsilon}  \ll_{k, \epsilon} \phi(q)(S+1)(\log q)^{k^2 + \epsilon} ,
\end{align*}
where the first inequality comes from 
\begin{align*}
\sum_{q_1|q}\phi(q_1) \prod_{\substack{p|q \\ p \nmid q_1}} \left(1 + \frac{1}{p^{1/2}}\right)^{4k} &= \prod_{p^e || q} \left[\left( 1 + \frac{1}{p^{1/2}}\right)^{4k} + \phi(p) + \phi(p^2) + ....+ \phi(p^{e})\right] 
\\ 
&= \prod_{p^e || q} \left[\left( 1 + \frac{1}{p^{1/2}}\right)^{4k} + p^e - 1\right] 
\\ 
&= \prod_{p^e || q} p^e\left( 1 + O\left( \frac{1}{p^{3/2}}\right)\right) \ll q,
\end{align*}
and the last inequality comes from $\phi(q) \gg \frac{q}{\log \log q}.$

The proof of (\ref{eqn:sumoverprimMoment}) is carried by the same arguments as the proof of Corollary A in \cite{Sound}, but here we  use the orthogonality relation in Lemma \ref{lem:orthogonal}  instead of Lemma 3 in \cite{Sound}.

\end{proof}
Similarly we have
\begin{prop} \label{thm:4twistedmomentIntandSumoverq} Assume GRH. For $q \geq 3$, $S > 0$  and for any positive real number $k$ and any $\epsilon > 0$ we have
\begin{align*}
& \sum_{\chi \ (\rm{mod} \ q)} \int_{S \leq |s| < 2S}\left|L\left( \frac{1}{2} + s, f\times \chi \right) \right|^{2k}\ll \phi(q) (S + 1) (\log q(S + 1))^{k^2 + \epsilon},
\end{align*}where the integration over $s$ is taken to be on a vertical line with $0\leq \tRe(s) \leq \frac{1}{\log Q}$.

\end{prop}

\section{Truncation} \label{sec:truncation}
From Lemma \ref{prop:fnceqnLambda}, we want to study the moment
$$  \M = \sum_{q } \Psi\left(\frac{q}{Q}\right) \sumb_{\chiq} \int_{-\infty}^{\infty} \left| \Lambda\left( \tfrac{1}{2} + iy, \chi \right)\right|^8 \> dy = 2\Delta(\Psi, Q),$$
where
$$ \Delta(\Psi, Q) = \sum_{q} \sumb_{\chiq} \Psi \left(\frac{q}{Q}\right) \Lambda_1(\chi).$$
By orthogonality as in Lemma \ref{lem:orthogonal}, we obtain that
$$\Delta(\Psi, Q) =  \h \sum_{m, n=1}^\infty \frac{\tau_4(m) \tau_4(n)}{\sqrt{mn}}\sum_{\substack{d, r\\(dr, mn) = 1 \\ r|m\pm n}} \phi(r) \mu(d) \Psi\bfrac{dr}{Q} V(m, n, dr).
$$
The first step to prove Theorem \ref{thm:eightmoment} is to slightly truncate the sums over $m$ and $n$ in $\Delta(\Psi, Q)$.  This type of procedure has appeared in the contexts of moments in other situations (see for instance, the works of Soundararajan \cite{SoundDiriMoment} and Soundararajan and Young \cite{SY}).  For a fixed $\alpha>0,$ define
 $$\dt (\Psi, Q) = \sum_{q} \sumb_{\chiq} \Psi \left(\frac{q}{Q}\right) \sum_{m, n = 1}^{\infty} \frac{\tau_4(m) \tau_4(n)}{\sqrt{mn}} \chi(m) \overline{\chi}(n) V\left(m, n ;\frac{q}{\lon}\right).$$ 
Note that we expect (and will later show) that $\Delta (\Psi, Q) \asymp Q^2(\log Q)^{16}.$ In this section, we prove that  $\dt (\Psi, Q)$ is a sufficiently close approximation to $\Delta (\Psi, Q)$.  This will allow us to apply the complementary divisor trick to $\dt$ and reduce the conductor in Section \ref{secG}.

\begin{prop} \label{prop:truncation}Preserve notation as above.  Then for any $\epsilon > 0,$
$$ {\Delta}(\Psi, Q) - \dt\left(\Psi, Q \right) \ll Q^2 (\log Q)^{15+\epsilon},$$
where the implied constant depends on $\epsilon$ and $\alpha.$
\end{prop}

The key ingredient for proving the proposition is the large sieve (Lemma \ref{prop:largesieve}).

\begin{proof} First, we introduce a smooth partition of unity.  Let $\sumd$ denote a dyadic sum, and let $F_M$ be a smooth function supported in $[M/2,3M]$ satisfying $F^{(j)}_{M}(x)\ll_j \frac{1}{M^j}$ for all $j\geq 0$, and such that 
$ 1 = \sumd_M F_M(x).$
Moreover the sum over $M$ is such that $\sumd_{M\leq X} 1 \ll \log X$. 

By the change of variables from $s$ to $\frac{v + z}{2}$ and $t$ to $\frac{v-z}{2i}$, and applying the smooth partition of unity, we have that $ {\Delta}(\Psi, Q) - \dt\left(\Psi, Q \right)$ is 
\begin{align} \label{eqn:boundtrunc}
&\frac{2\pi}{(2\pi i)^2} \sumd_M \sumd_N \int_{(1)} \int_{(1)} \Gamma^4\left( \frac{1}{4} + \frac{v}{2} \right) \Gamma^4 \left(\frac{1}{4} + \frac{z}{2}\right)\left[ \left( \frac{Q}{\pi}\right)^{2v + 2z} - \left( \frac{Q}{\pi\lon}\right)^{2v + 2z}\right] \cdot \nonumber \\
&\hskip .5in \cdot  \sum_{q} \Psi\left(\frac{q}{Q}\right) \left( \frac{q}{Q}\right)^{2v + 2z}\sumb_{\chi \ {\rm mod}\ q}\sum_{m}^{\infty} \frac{\tau_4(m)\chi(m)}{m^{1/2 + v}} F_M(m) \sum_{n}^{\infty}\frac{\tau_4(n)\overline{\chi(n)}}{n^{1/2 + z}} F_N(n)\> \frac{dv \ dz}{(v+z)}.   
\end{align}

To evaluate the integrals above, we shift contour integrals according to the size of $M$ and $N$. After shifting the integrals, we bound the sum over $q$ and $\chi$ by using Cauchy-Schwarz inequality and then applying the large sieve.  In particular, we have
\begin{align*}
& \sum_{q} \Psi\left(\frac{q}{Q}\right) \sumb_{\chi \ {\rm mod}\ q}\sum_{m}^{\infty} \frac{\tau_4(m)\chi(m)}{m^{1/2 + v}} F_M(m) \sum_{n}^{\infty}\frac{\tau_4(n)\overline{\chi(n)}}{n^{1/2 + z}} F_N(n) \\
&\ll \left(\sum_{q} \Psi\left(\frac{q}{Q}\right) \sumb_{\chi \ {\rm mod}\ q}\left| \sum_{m}^{\infty} \frac{\tau_4(m)\chi(m)}{m^{1/2 + v}} F_M(m)\right|^{2} \right)^{\frac 12} \left(\sum_{q} \Psi\left(\frac{q}{Q}\right) \sumb_{\chi \ {\rm mod}\ q}\left| \sum_{n}^{\infty}\frac{\tau_4(n)\overline{\chi(n)}}{n^{1/2 + z}} F_N(n) \right|^2 \right)^{\frac 12} \\
&\ll (Q^2 + M)^{\frac 12} (Q^2 + N)^{\frac 12}\left| \sum_{M/2 \leq n\leq 3M}\frac{\tau_4^2(m)}{m^{1 + 2 \R (v)}}  \right|^{\frac 12}\left| \sum_{N/2 \leq n\leq 3N}\frac{\tau_4^2(n)}{n^{1 + 2 \R (z)}}  \right|^{\frac 12} \\
&\ll (Q^2 + M)^{\frac 12} (Q^2 + N)^{\frac 12} \frac{1}{M^{\R v}N^{\R z}}(\log M)^{\frac{15}{2}} (\log N)^{\frac {15}{2}}.
\end{align*}
Now, we consider the following cases.
\subsection*{Case 1: $M > Q^2$} (By symmetry, the case $N > Q^2$ is the same.) We do not shift the integral over $v$. The contour integration over $z$ is shifted to the line $\R(z) = \ell$, where $\ell = -1/4$ if $N \leq Q^2$, and $\ell = 1$ otherwise. For fixed $M$ and $N$, the integral appearing in \eqref{eqn:boundtrunc} is bounded by 
\begin{align*}
&\ll \int_{(1)} \int_{(\ell)}  \left| \Gamma^{4} \left( \frac{1}{4} + \frac{v}{2}\right)\Gamma^{4} \left( \frac{1}{4} + \frac{z}{2}\right)\right| \frac{Q^{2 + 2\ell} }{MN^\ell} (Q^2 + M)^{\frac 12}(Q^2 + N)^{\frac 12} (\log M)^{\frac{15}{2}}(\log N)^{\frac {15}{2}} \> d|v| \> d|z| \\
&\ll \frac{Q^{2 + 2\ell} }{MN^\ell}  M^{\frac 12}(Q^2 + N)^{\frac 12} (\log M)^{\frac{15}{2}}(\log N)^{\frac {15}{2}}.
\end{align*}
Hence the contribution to (\ref{eqn:boundtrunc}) is 
\begin{align*}
&\ll \sumd_{M > Q^2} \sumd_{N > Q^2} \frac{Q^4}{MN} M^{\frac 12}N^{\frac 12} (\log M)^{\frac{15}{2}}(\log N)^{\frac{15}{2}} + \sumd_{M > Q^2} \sumd_{N \leq Q^2} \frac{Q^{\frac 32}N^{\frac 14}}{M} M^{\frac 12}Q (\log M)^{\frac{15}{2}}(\log N)^{\frac{15}{2}}  \\
&\ll Q^2(\log Q)^{15}.
\end{align*}

\subsection*{Case 2: $\frac{Q^2}{\lotwo} < M \leq Q^2$ and $N \leq Q^2$} (By symmetry, this is the same case as when $\frac{Q^2}{\lotwo} < N \leq Q^2$ and $M\leq Q^2$). We shift the contour integral over $v$ to $\R(v) = 0.$ Moreover we shift the contour in $z$ to $\R(z) = \ell$ where $\ell = 0$ if $\frac{Q^2}{\lotwo} < N \leq Q^2$ and $\ell = -1/4$ otherwise. We remark that 
$$ \frac{1}{v+ z}\left[\left( \frac{Q^2}{\pi^2}\right)^{v+ z} - \left(\frac{Q^2}{\pi^2\lotwo }\right)^{v + z}\right] $$
is entire in $v$ and $z$. For fixed $M$ and $N$, the integral appearing in \eqref{eqn:boundtrunc} is bounded by 
$
Q^2(\log Q)^{15 + \epsilon}\frac{Q^{2\ell}}{(\log Q)^{2\alpha \ell}N^\ell}.
$ Hence the contribution to (\ref{eqn:boundtrunc}) from this case is 
\begin{align*}
&\ll \sumd_{\frac{Q^2}{\lotwo} < M \leq Q^2} \ \sumd_{\frac{Q^2}{\lotwo } < N \leq Q^2} Q^2 (\log Q)^{15 + \epsilon}  + \sumd_{\frac{Q^2}{\lotwo} < M \leq Q^2} \ \sumd_{N \leq \frac{Q^2}{\lotwo}} Q^{\frac 32} (\log Q)^{15 + \epsilon + \frac{\alpha}{2}} N^{\frac 14} \\
&\ll Q^2 (\log Q)^{15 + \epsilon}.
\end{align*}

\subsection*{Case 3: $M \leq \frac{Q^2}{\lotwo}$ and $N \leq \frac{Q^2}{\lotwo}$} For this case we shift integrals in both $v$ and $z$ to the line with real part $-1/4$. Since the integrand is holomorphic, we encounter no poles. The contribution to (\ref{eqn:boundtrunc}) from this case is 
\begin{align*}
&\ll  Q^2 (\log Q)^{15} \sumd_{ M \leq \frac{Q^2}{\lotwo}} \ \sumd_{N \leq \frac{Q^2}{\lotwo}} \frac{M^{\frac 14}N^{\frac 14} (\log Q)^{\alpha}}{Q} \ll Q^2 (\log Q)^{15}. 
\end{align*}
Case 1 - Case 3 gives that Equation \ref{eqn:boundtrunc} is bounded by $Q^2(\log Q)^{15 + \epsilon}.$
\end{proof}

\section{Splitting off the diagonal terms}
From the previous section, it is sufficient to consider
$$\dt\left(\Psi, Q\right) = \h \sum_{m, n=1}^\infty \frac{\tau_4(m) \tau_4(n)}{\sqrt{mn}}\sum_{\substack{d, r\\(dr, mn) = 1 \\ r|m\pm n}} \phi(r) \mu(d) \Psi\bfrac{dr}{Q} V\left(m, n, \frac{dr}{\lon}\right).$$
Let $D = (\log Q)^\delta$ for $\delta>0$ a parameter to be determined (eventually we pick $\delta = 130$) and split 
\begin{equation}\label{eqn:split}
\dt\left(\Psi, Q\right) = \D  +\Sm + \G
\end{equation}
where the diagonal term $\D$ consists of the terms with $m=n$, the term $\Sm$ consists of the remaining terms with $d>D$ and $\G$ consists of the rest of the terms with $d \leq D$. 

The diagonal terms $\D$ can be computed as in the following Proposition
\begin{prop}\label{sec:diagonal}
Let $\widetilde{\Psi}$ be the Mellin transform of $\Psi,$ which is defined by \begin{equation} \label{eqn:MellinPsi}
 \widetilde{\Psi} (s) = \int_{0}^\infty \Psi(u) u^s \frac{du}{u}.
\end{equation}
Then 
\begin{align} \label{eqn:diag1}
\D &= 2^{16}\sum_{q} \phib \Psi\bfrac{q}{Q} \frac{(\log q )^{16}}{16!} \frac{A(1/2)}{B_q(1/2)} \int_{-\infty}^{\infty} G(1/2, t)dt + O(Q^2(\log Q)^{15+ \epsilon}).
\end{align}
Moreover we can write the sum of the main term in (\ref{eqn:diag1}) as
\begin{align} \label{eqn:eulerfordiagonal}
 2^{16}Q^2 \frac{(\log Q )^{16}}{16!} \widetilde{\Psi}(2) \frac{A(1/2)}{2} \prod_p\left( 1 - \frac{1}{p}\right)\left(1 + \frac{1}{B_p(1/2)} \left( \frac{1}{p} - \frac{1}{p^2} - \frac{1}{p^3}\right)\right) \int_{-\infty}^{\infty} G(1/2, t)dt.
\end{align}

\end{prop}

\begin{proof}

We have that
\begin{align*}
 \D &= \h \sum_{n=1}^\infty \frac{\tau_4^2(n) }{n}\sum_{\substack{d, r\\(dr, n) = 1}} \phi(r) \mu(d) \Psi\bfrac{dr}{Q} V\left(n, n, \frac{dr}{\lon }\right)\\
&= \sum_{q} \phib \Psi\bfrac{q}{Q} \frac{1}{2\pi i} \int_{-\infty}^\infty \int_{(1)}  \left(\sum_{\substack{n=1 \\ (n, q) = 1}}^\infty \frac{\tau_4^2(n)}{n^{1+2s}} \right) G(1/2+s, t) \bfrac{q}{2\pi \lon }^{4s}\frac{ds}{s} dt.
\end{align*}
From Lemma \ref{lem:eulerprod}, 
$$ \sum_{n=1}^\infty \frac{\tau_4^2(n)}{n^{1+2s}} = \zeta^{16}(1 + 2s) \frac{A(s + 1/2) }{\mathcal B_q(s + 1/2)},$$
where $A$ is analytic when $\tRe s > -1/4.$

Thus the integrand has a pole of order $17$ at $s=0$ in the region $\tRe s > -1/4 + \epsilon$, and shifting $s$ to $\tRe (s) = -1/4 + \epsilon$ gives a residue of 
$$ \frac{4^{16}}{2^{16} 16!} \frac{A(1/2)}{B_q(1/2)} G(1/2, t) (\log q)^{16} + O((\log q)^{15+ \epsilon}),
$$
and we obtain (\ref{eqn:diag1}). 

To obtain (\ref{eqn:eulerfordiagonal}), we use the fact that $\phib = \frac{1}{2} \phi^*(q) + O(1),$ and the function $\phi^*$ is multiplicative with $\phi^*(p) = p-2$ and $\phi^*(p) = p^{k-2}(p-1)^2$ for $k \geq 2.$

\end{proof}
\section{The off-diagonal term: the sum $\Sm$} \label{sec:calS}
Recall that
$$\Sm = \h\sum_{\substack{m, n=1\\m\neq n}}^\infty \frac{\tau_4(m) \tau_4(n)}{\sqrt{mn}}\sum_{\substack{d, r\\(dr, mn) = 1 \\ r|m\pm n\\d>D}} \phi(r) \mu(d) \Psi\bfrac{dr}{Q} V\left(m, n, \frac{dr}{\lon}\right).
$$

We express the condition $r| m \pm n$ using the even characters $\chir.$ Specifically,
$$ \Sm = \sum_{\substack{d, r\\d>D}} \mu(d) \Psi\bfrac{dr}{Q} \sum_{\substack{\chir \\ \chi(-1) = 1}}  \sum_{\substack{m, n=1 \\ m \neq n \\(d, mn) = 1 }}^\infty \frac{\chi(m)\overline{\chi}(n)\tau_4(m) \tau_4(n)}{\sqrt{mn}}\V\left(m, n, \frac{dr}{\lon}\right).$$

The contribution to the principal character $\chi = \chi_0$ gives a main term, and the non-principal characters contribute to an acceptable error term.

\begin{prop} \label{prop:sumSm} There exists an absolute constant $A$ such that
$$ \Sm = \mathcal {MS}(\Psi, Q) + O\left(\frac{Q^2 (\log Q)^{16 + A + \epsilon}}{D^{1 - \epsilon}}\right),$$
where 
\begin{equation} \label{eqn:mainMS}
 \mathcal {MS}(\Psi, Q) = - \sum_{\substack{m, n = 1  \\ m \neq n}} \frac{\tau_4 (m) \tau_4 (n)}{\sqrt{mn}} \sum_{(q, mn) = 1} \Psi \left( \frac{q}{Q}\right)\left( \sum_{\substack{dr = q \\ d \leq D}} \mu(d) \right)  V\left(m, n ; \frac{q}{\lon} \right).
\end{equation}

\end{prop}
Note that taking $D > (\log Q)^{A+1}$ gives us an acceptable error term.  

\begin{proof}
The principal character gives 
$$ \sum_q \left( \sum_{\substack{dr = q \\d>D}} \mu(d)\right) \Psi\bfrac{q}{Q}  \sum_{\substack{m, n=1 \\ m \neq n \\(q, mn) = 1 }}^\infty \frac{\tau_4(m) \tau_4(n)}{\sqrt{mn}}\V\left(m, n, \frac{q}{\lon}\right).$$
Since $\sum_{dr = q} \mu(d) = 0$ for $q > 1$, the above equals to $\mathcal {MS}$ stated in the proposition. 

Now we consider the contribution to the non-principal characters. We first reintroduce the terms $m=n$.  This gives an acceptable error since
\begin{align*}
&\sum_{n=1}^\infty \frac{\tau_4^2(n)}{n}\sum_{\substack{d, r\\(dr, n) = 1 \\d>D}} \phi(r) \mu(d) \Psi\bfrac{dr}{Q} \V\left(n, n, \frac{dr}{\lon}\right)\\
&= \sum_{\substack{d, r\\d>D}} \phi(r) \mu(d) \Psi\bfrac{dr}{Q}  \frac{1}{2\pi i} \int_{(1)} \int_{-\infty}^\infty \sum_{\substack{n=1\\(dr, n) = 1} }^\infty \frac{ \tau_4^2(n)}{n^{1+2s}} G(1/2+s, t) \bfrac{dr}{2\pi \lon }^{4s}\frac{ds}{s} dt. 
\end{align*}
From the calculation of the diagonal term, the double integral is $\ll (\log dr)^{16}.$ Therefore the above is

\begin{align*}
&\ll (\log Q)^{16} \sum_{\substack{d, r\\d>D}} \phi(r) \Psi\bfrac{dr}{Q} \ll  (\log Q)^{16} \sum_{d>D} \sum_{r<2Q/d} r \ll  \frac{Q^2(\log Q)^{16}}{D}.
\end{align*}

Now we bound the resultant sum
\begin{align*}
\Smo &= \sum_{\substack{d, r\\d>D}} \mu(d) \Psi\bfrac{dr}{Q} \sum_{\substack{\chir \\ \chi(-1) = 1 \\ \chi \neq \chi_0}}  \sum_{\substack{m, n=1\\(d, mn) = 1 }}^\infty \frac{\chi(m)\overline{\chi}(n)\tau_4(m) \tau_4(n)}{\sqrt{mn}}\V\left(m, n, \frac{dr}{\lon}\right)\\
&= \sum_{\substack{d, r\\d>D}} \mu(d) \Psi\bfrac{dr}{Q} \sum_{\substack{\chir \\ \chi(-1) = 1 \\ \chi \neq \chi_0}}  \frac{1}{2\pi i} \int_{(1)} \int_{-\infty}^\infty \sum_{\substack{m, n=1\\(d, mn) = 1}}^\infty \frac{\chi(m)\overline{\chi}(n)\tau_4(m) \tau_4(n)}{m^{1/2+s + it}n^{1/2+s - it}}  \cdot \\
& \hskip 3in  \cdot G(1/2+s, t)  \bfrac{dr}{2\pi(\log Q)^{\alpha} }^{4s}\frac{ds}{s} dt.
\end{align*}
We have
$$\sum_{\substack{m, n=1\\(d, mn)=1}}^\infty \frac{\chi(m)\overline{\chi}(n)\tau_4(m) \tau_4(n) }{m^{1/2+s + it}n^{1/2 + s - it}}=  \frac{L^4(1/2+s + it,  \chi)L^4(1/2+s - it,  \overline{\chi})}{L_d^4(1/2+s + it,  \chi)L_d^4(1/2+s -it, \cb)},
$$
where $L_d(s, \psi) = \prod_{p|d} \left(1 - \frac{\psi(p)}{p^{s}} \right)^{-1} \ll d^{\epsilon}$. Since $\chi$ is non-principal, we can shift the line of integration to $\tRe (s) = \frac{1}{\log Q}$ without passing any poles. We write $s = \frac{1}{\log Q} + iv,$ and let $t_1 = v + t$ and $t_2 = v - t$. Also, we can write $ \sum_{\chir} = \sum_{\ell | r} \sumstar_{\chi' \  \textrm{mod} \ \ell  }$, and 
\begin{align*}
 L\left(\tfrac{1}{2}+ \tfrac{1}{\log Q} + it,  \chi\right) &= L\left(\tfrac{1}{2}+ \tfrac{1}{\log Q} + it,  \chi'\right) \prod_{\substack{p | r \\ p \nmid \ell}} \left(1 - \frac{\chi'(p)}{p^{\frac{1}{2} + \frac{1}{\log Q} + it}}\right) \\
&\ll \tau\left( \frac{r}{\ell}\right)\left| L\left(\tfrac{1}{2}+ \tfrac{1}{\log Q} + it,  \chi'\right)\right|
\end{align*}
Then, since $ G\left(1/2 + s, t\right) \ll \exp(-|t_1| - |t_2|)$, we apply Cauchy-Schwarz to obtain
\begin{align*}
\Smo &\ll \log Q \sum_{d > D} d^{\varepsilon} \int_{-\infty}^{\infty} \int_{-\infty}^{\infty} \exp(-|t_1| - |t_2|) \sum_{ \ell \leq \frac{2Q}{d}} \sum_{\substack{r \\\ell | r}} \tau^{8}\left( \frac{r}{\ell}\right)\times \\
& \times \sumstar_{\substack{\chi' \textrm{mod} \ \ell \\ \chi(-1) = 1}} \left\{ \left| L\left(\tfrac{1}{2}+ \tfrac{1}{\log Q} + it_1,  \chi'\right)\right|^8 + \left|L\left(\tfrac{1}{2}+ \tfrac{1}{\log Q} + it_2,  \overline{\chi'}\right)\right|^8 \right\} \> dt_1 \> dt_2.
\end{align*}
We have $\sum_{\substack{r \\\ell | r}} \tau^{8}\left( \frac{r}{\ell}\right) \ll \sum_{h \leq \frac{2Q}{d\ell}} \tau^8(h) \ll \frac{2Q}{d\ell} (\log Q)^A$ for an absolute constant $A$, where $\tau$ is the divisor function. Using this and  Lemma \ref{prop:LfunctionofLargesieve}, we get
\begin{align*}
\Smo &\ll Q(\log Q)^{A}\sum_{d > D} \frac{1}{d^{1 - \varepsilon}} \int_{-\infty}^{\infty} \int_{-\infty}^{\infty} \exp(-|t_1| - |t_2|) \sumd_{L \leq \frac{2Q}{d}} \frac{1}{L}\sum_{\ell \sim L } \\
&\times \sumstar_{\substack{\chi' \textrm{mod} \ \ell \\ \chi(-1) = 1}} \left\{\left| L\left(\tfrac{1}{2}+ \tfrac{1}{\log Q} + it_1,  \chi'\right)\right|^8 + \left|L\left(\tfrac{1}{2}+ \tfrac{1}{\log Q} + it_2,  \overline{\chi'}\right)\right|^8 \right\} \> dt_1 \> dt_2\\
&\ll Q(\log Q)^{A}\sum_{d > D} \frac{1}{d^{1 - \varepsilon}} \sumd_{L \leq \frac{2Q}{d}} L \ll \frac{Q^2 (\log Q)^A}{D^{1 - \epsilon}}, 
\end{align*}
where $A$ is an absolute constant. 

\end{proof}
\section{Treatment of $\G$}\label{secG}
\subsection{The complementary divisor}
Recall that
$$\G = \h \sum_{\substack{m,n=1\\m\neq n}}^{\infty} \frac{\tau_4(m)\tau_4(n)}{\sqrt{mn}} \sum_{\substack{d, r\\(dr, mn)=1\\r|m\pm n\\d\leq D}} \mu(d) \phi(r) \Psi\bfrac{dr}{Q} \V\left(m, n, \frac{dr}{\lon}\right).
$$
We write $g=(m, n)$ and $m=gM$, $n=gN$, so that $(M, N) = 1$, and write $\phi(r) = \sum_{al=r}\mu(a)l.$  The latter identity extracts the arithmetic information from $\phi(r)$ so that we may eventually sum smoothly over the essential parts of the modulus.  Then the sum over $d$ and $r$ is
$$\sum_{\substack{d, a, l\\(d, mn)=1\\(al, g)=1\\al|M\pm N\\d\leq D}} \mu(d) \mu(a) l \Psi\bfrac{dal}{Q} \V\left(gM, gN, \frac{dal}{\lon}\right).
$$

In the above, we have used that $(r, mn) = 1$ if and only if $ (r, g) = 1$ since $r|m \pm n$.  We let $|M\pm N| = alh$.  We want to replace the condition modulo $r=al$ with a condition modulo $h$, which will be small when $r$ is large.  Thus we replace $l$ with $\frac{|M \pm N|}{ah}$.  To do so, we express the condition $(l, g) = 1$ by $\sum_{b|(l, g)} \mu(b)$.  Writing $l = bk$, the sum becomes
$$\sum_{\substack{d\leq D\\(d, gMN) = 1}} \mu(d) \sum_{(a, g) = 1} \mu(a) \sum_{b|g}\mu(b) \sum_{\substack{k\geq 1\\|M\pm N| = abkh}}bk \Psi\bfrac{dabk}{Q} \V\left(gM, gN, \frac{dabk}{\lon }\right).
$$
We substitute $k=\frac{|M \pm N|}{abh}$ to get 
\begin{align} \label{eqn:stepbeforeW}
Q \sum_{d\leq D}  \sum_{(a, g)=1} \sum_{b|g} \sum_{\substack{h > 0 \\ M \equiv \mp N \ ({\rm mod} \ abh )}} & \frac{\mu(d) \mu(a) \mu(b) }{ad} \\
& \times \left(\frac{d|M\pm N|}{Qh} \right)\Psi \bfrac{d|M \pm N|}{Qh} \V\left(gM, gN; \frac{d|M \pm N|}{h Q \lon }\right). \nonumber
\end{align}
For non-negative real numbers $u, x, y$ and for each choice of sign, we define
$$\W^\pm(x, y; u) = u|x \pm y|\Psi(u|x\pm y|) \V(x, y; u|x\pm y|).
$$
Since
$$\V(m, n; \mu) =  \int_{-\infty}^\infty \bfrac{n}{m}^{it} W\left( mn\bfrac{\pi}{\mu}^4, t\right) dt,
$$
we have
\begin{equation} \label{eqn:Vc}
\V(cm, cn; \sqrt c \mu) 
= \int_{-\infty}^\infty \bfrac{n}{m}^{it} W\left( mn\bfrac{\pi}{\mu}^4, t\right) dt  = \V (m, n; \mu).
\end{equation}
Thus (\ref{eqn:stepbeforeW}) becomes
\begin{align*}
& Q \sum_{d\leq D}  \sum_{(a, g)=1} \sum_{b|g} \sum_{\substack{h > 0 \\ M \equiv \mp N \ ({\rm mod} \ abh )}}  \frac{\mu(d) \mu(a) \mu(b) }{ad} \W^\pm \left(\frac{gM \lotwo}{Q^2}, \frac{gN \lotwo}{Q^2} ; \frac{Qd}{gh \lotwo}\right). 
\end{align*}
Note that $(MN, abh) = 1.$  We express the condition $M \equiv \mp N (\ {\rm mod} \ abh)$ using characters $\chi$ (mod $abh$).  We then separate the principal character contribution, which is the main term, and the non-principal characters which contribute to an acceptable error term. Specifically,
$$ \G = \mathcal{MG}(\Psi, Q) + \mathcal{EG}(\Psi, Q),$$
where 
\begin{align} \label{eqn:maintermG}
\mathcal{MG}(\Psi, Q) = \frac{Q}{2} \sum_{\substack{m, n = 1\\ m \neq n}}^{\infty} & \frac{\tau_4(m) \tau_4(n)}{\sqrt{mn}}\sum_{\substack{d \leq D \\ (d, gMN) = 1}}  \sum_{(a, g)=1} \sum_{b|g} \sum_{\substack{h > 0 \\ (abh, MN) = 1}}  \frac{\mu(d) \mu(a) \mu(b) }{ad \phi(abh)}\\
&\times  \W^\pm \left(\frac{gM \lotwo}{Q^2}, \frac{gN \lotwo}{Q^2} ; \frac{Qd}{gh \lotwo}\right), \nonumber
\end{align}
and
\begin{align} \label{eqn:errorG}
\mathcal{EG}(\Psi, Q) =  \frac{Q}{2} \sum_{\substack{m, n = 1\\ m \neq n}}^{\infty} & \frac{\tau_4(m) \tau_4(n)}{\sqrt{mn}}\sum_{\substack{d \leq D \\ (d, gMN) = 1}}  \sum_{(a, g)=1} \sum_{b|g} \sum_{\substack{h > 0 \\ (abh, MN) = 1}}  \frac{\mu(d) \mu(a) \mu(b) }{ad \phi(abh)} \\
&\times \sum_{\substack{\chi \ ({\rm mod} \ abh) \\ \chi \neq \chi_0}} \chi(M) \cb(\mp N) \W^\pm \left(\frac{gM \lotwo}{Q^2}, \frac{gN \lotwo}{Q^2} ; \frac{Qd}{gh \lotwo}\right). \nonumber
\end{align}

\subsection{Mellin transforms of $\W^\pm$} \label{sec:Mellin} A simple heuristic estimate of $\mathcal{MG}$ tells us that it is, up to a factor of a power of $\log Q$, of size $Q^2$ (for this estimate, note that $h$ is the essential part of the modulus $abh$).  Since we expect cancellation in the sum when $\chi$ is non-principal, we expect $\mathcal{EG}$ to be small.

To evaluate $\Mg$ and $\Eg$ more precisely, we write $\W^\pm$ in terms of its Mellin transforms. There are three different types that we shall consider. They come from taking Mellin transforms in the variable $u$ for when we need to sum over the modulus $h$, the variables $x , y$ for when we need to sum over $M$ and $N$, and all three variables for when we need to sum over $M$, $N$, and $h$.  (In the description above, we have neglected to mention the conceptually less important sums over $d$, $a$, $b$ and $g$.)  

We collect the properties of the various Mellin transforms in the following three lemmas. The proofs of the lemmas are the same as  the ones in Section 7 of \cite{CIS}, but using the bounds of Lemma \ref{lem:weightWV} instead. 
\begin{lemma} \label{lem:MellinU}
Given positive real numbers $x$ and $y$, define
$$ \Wt^\pm_1(x, y; z) = \int_0^\infty \W^\pm (x, y ; u) u^{z} \frac{du}{u}.$$
Then the functions $\Wt^\pm_1(x, y; z)$
are analytic for all $z \in \mathbb C$. We have the Mellin inversion formula
\begin{equation} \label{eqn:mellinU}
 \W^\pm(x, y; u) = \frac{1}{2\pi i} \int_{(c)} \Wt^\pm_1(x, y; z) u^{-z} \> dz,
\end{equation}
where the integral is taken over the line $\tRe(z) = c$ for any real number $c$. The Mellin transforms $\Wt^\pm_1(x, y; z)$ satisfy for any non-negative integer $\nu$
$$ |\Wt^\pm_1(x, y;z)| \ll_\nu |x \pm y|^{-\tRe z} \prod_{j = 1}^\nu |z + j|^{-1} \exp \left(-c_1 \max (x, y)^{1/4}\right)$$
for some absolute constant $c_1.$
\end{lemma}

\begin{lemma} \label{lem:MellinXY}
Given a positive real number $u,$ we define
$$ \Wt^\pm_2 (s_1, s_2; u) = \int_0^\infty \int_0^\infty \W^\pm(x, y ; u) x^{s_1}y^{s_2} \frac{dx}{x} \frac{dy}{y}.$$
Then the functions $\Wt^\pm_2(s_1, s_2 ; u)$ are analytic in the region $\tRe (s_1), \tRe(s_2) > 0$. We have the Mellin inversion formula
$$ \W^\pm(x, y ; u) = \frac{1}{(2\pi i)^2} \int_{(c_1)}\int_{(c_2)} \Wt^\pm_2 (s_1, s_2 ; u) x^{-s_1} y^{-s_2} \>d s_1 \> d s_2,$$
where $c_1, c_2$ are positive. The Mellin transforms $\Wt^\pm_2(s_1,s_2 ; u)$ satisfy, for any $k \geq 1$
$$ |\Wt^\pm_2(s_1, s_2; u)| \ll \frac{(1 + u)^{k-1}}{\max(|s_1|, |s_2|)^k} \exp\left(-c_1 u^{-1/4}\right).$$
\end{lemma}

\begin{lemma} \label{lem:MellinXYU} We define
$$ \Wt^\pm_3 (s_1, s_2; z) = \int_0^\infty \int_0^\infty \W^\pm(x, y ; u) u^z x^{s_1}y^{s_2} \frac{du}{u} \frac{dx}{x} \frac{dy}{y},$$
and 
$$ \Wt_3 (s_1, s_2; z) = \Wt^+_3 (s_1, s_2; z) + \Wt^-_3 (s_1, s_2; z).$$
Let $\omega = \frac{s_1 + s_2 - z}{2}$ and $\xi = \frac{s_1 - s_2 + z}{2}.$ For $\tRe(s_1), \tRe(s_2) > 0,$ and $|\tRe(s_1 -s_2)| < \tRe(z) < 1$ we have
\begin{equation} \label{eqn:mellinthree}
\Wt_3(s_1, s_2; z) = \frac{\psit(1 + 4\omega + z)}{2\omega \pi^{4\omega}} \int_{-\infty}^{\infty} \Hc (\xi - it, z) G\left(\h + \omega, t \right) \> dt, 
\end{equation}
where $\psit$ is defined in (\ref{eqn:MellinPsi}), and 
$$ \Hc (u, v) = \pi^{1/2} \frac{\Gamma\left(\tfrac{u}{2} \right)\Gamma\left(\tfrac{1-v}{2} \right)\Gamma\left(\tfrac{v-u}{2} \right)}{\Gamma\left(\tfrac{1-u}{2} \right)\Gamma\left(\tfrac{v}{2} \right)\Gamma\left(\tfrac{1-v + u}{2} \right)}.$$
Moreover we have the Mellin inversion formula
$$ \W^\pm(x, y ; u) = \frac{1}{(2\pi i)^3} \int_z \int_{s_1}\int_{s_2} \Wt^\pm_2 (s_1, s_2 ; z) u^{-z} x^{-s_1} y^{-s_2} \>d s_2 \> d s_1 \> dz, $$
where all of the paths are taken to be the vertical lines with increasing imaginary parts and real parts satisfying the constraints given above, and the integrals over $s_1$ and $s_2$ are to be interpreted as being over $|{\rm Im} (s_1)| \leq T_1$ and $|{\rm Im} (s_s)| \leq T_s$ and letting $T_1, T_2$ tend to infinity. Finally the Mellin transform $\Wt_3(s_1, s_2;z)$ satisfies the bound
\begin{equation} \label{eqn:boundWt3}
|\Wt_3(s_1, s_2;z)| \ll (1 + |z|)^{-A} (1 + |\omega|)^{-A} (1 + |\xi|)^{\tRe(z) - 1}.
\end{equation}
\end{lemma}
\section{Bounding the error term $\Eg$} \label{sec:errorEg}
In this section and later, we write $\beta$ for the absolute positive constant, which may stand for different values from line to line. 

First we show that we can restrict the sum over $a$ to $a \leq 2Q$. Since  $M \neq N$, if $a > 2Q$ and $M \equiv \mp N$ (mod $abh$) then
$$ \frac{Q d}{gh \lotwo} \frac{|gM \pm gN| \lotwo}{Q^2} \geq \frac{dab}{Q} \geq 2,$$ 
so $\W^\pm\left( \frac{gM \lotwo}{Q^2}, \frac{gN \lotwo}{Q^2}; \frac{Qd}{gh\lotwo}\right) = 0$. Therefore

\begin{align*} 
\mathcal{EG}(\Psi, Q) =  \mathcal{EG}_1(\Psi, Q) - \mathcal{EG}_2(\Psi, Q), 
\end{align*}
where
\begin{align*}
\mathcal{EG}_1(\Psi, Q) = \frac{Q}{2} \sum_{\substack{m, n = 1\\ m \neq n}}^{\infty} & \frac{\tau_4(m) \tau_4(n)}{\sqrt{mn}}\sum_{\substack{d \leq D \\ (d, gMN) = 1}}   \sum_{\substack{(a, g)=1 \\ a \leq 2Q}} \sum_{b|g} \sum_{\substack{h > 0 \\ (abh, MN) = 1}}  \frac{\mu(d) \mu(a) \mu(b) }{ad \phi(abh)} \\
&\times \sum_{\substack{\chi \ ({\rm mod} \ abh) \\ \chi \neq \chi_0}} \chi(M) \cb(\mp N) \W^\pm \left(\frac{gM \lotwo}{Q^2}, \frac{gN \lotwo}{Q^2} ; \frac{Qd}{gh \lotwo}\right),
\end{align*}
and 
\begin{align*}
\mathcal{EG}_2(\Psi, Q) = \frac{Q}{2} \sum_{\substack{m, n = 1\\ m \neq n}}^{\infty} & \frac{\tau_4(m) \tau_4(n)}{\sqrt{mn}}\sum_{\substack{d \leq D \\ (d, gMN) = 1}}   \sum_{\substack{(a, g)=1 \\ a > 2Q}} \sum_{b|g} \sum_{\substack{h > 0 \\ (abh, MN) = 1}}  \frac{\mu(d) \mu(a) \mu(b) }{ad \phi(abh)} \\
&\times  \W^\pm \left(\frac{gM \lotwo}{Q^2}, \frac{gN \lotwo}{Q^2} ; \frac{Qd}{gh \lotwo}\right).
\end{align*}
Since $\Psi$ is supported in $[1,2],$ $\frac{|M \pm N|d}{Qh} \in [1, 2].$ From Lemma \ref{lem:weightWV}, we have that
$$\W^\pm \left(\frac{gM \lotwo}{Q^2}, \frac{gN \lotwo}{Q^2} ; \frac{Qd}{gh \lotwo}\right) \ll \exp\left( - c \left( \frac{\max (m, n) \lotwo}{Q^2}\right)\right),$$
so $\W^\pm \left(\frac{gM \lotwo}{Q^2}, \frac{gN \lotwo}{Q^2} ; \frac{Qd}{gh \lotwo}\right)$ is small unless $m, n \leq \frac{Q^2}{\lon}.$
Moreover, $\phi(abh) \gg abh \log \log (abh),$ and $\sum_{m \leq x} \frac{\tau_4(m)}{\sqrt m} \ll \sqrt x (\log x)^3.$ Therefore
\begin{align}  \label{eqn:EG2}
\begin{split}
\mathcal{EG}_2(\Psi, Q)  &\ll Q (\log Q)^{\epsilon}\sum_{m, n \leq \frac{Q^2}{\lon}} \frac{\tau_4(m) \tau_4(n)}{\sqrt{m n}} \sum_{d \leq D} \sum_{a > 2Q} \sum_{b \leq Q^2} \sum_{h \leq Q/d} \frac{1}{a^2bdh}  \\
&\ll Q(\log Q)^{\epsilon} \sum_{m, n \leq \frac{Q^2}{\lon}} \frac{\tau_4(m) \tau_4(n)}{\sqrt{m n}}\sum_{a > 2Q} \frac{1}{a^2} \ll \frac{Q^2}{(\log Q)^{\alpha - 1}},
\end{split}
\end{align}
which is an acceptable error term. Next we will show that $\mathcal{EG}_1(\Psi, Q)$ gives an acceptable error term as in the following Lemma.

\begin{lem} \label{lem:EG1bound} Assume GRH. We have 

$$\mathcal{EG}_1(\Psi, Q) \ll \frac{Q^2}{\log Q}.$$
\end{lem}

\begin{proof}
 By Lemma \ref{lem:MellinXY}, we can write $\mathcal{EG}_1(\Psi, Q)$ as
\begin{align*}
\frac{Q}{2} \sum_{\substack{a \leq 2Q \\ b, h > 0}} \sum_{\substack{\chi \ ({\rm mod} \ abh) \\ \chi \neq \chi_0}}& \sum_{\substack{g \\ b|g, (a,g) = 1}} \sum_{\substack{d \leq D \\ (d, g) = 1}} \frac{\mu(a) \mu(b) \mu(d)}{adg\phi(abh)} \frac{1}{(2\pi i)^2} \int_{(1/2 + \epsilon)} \int_{(1/2 + \epsilon)} \Wt^\pm_2 \left(s_1, s_2; \frac{Qd}{gh \lotwo}\right) \\
& \times \left(\frac{Q^2}{g \lotwo}\right)^{s_1 + s_2} \sum_{\substack{M, N = 1 \\ M \neq N, (M, N) = 1 \\ (MN, d) = 1} }^{\infty} \frac{\tau_4(gM)\tau_4(gN)}{M^{1/2 + s_1} N^{1/2 + s_2}} \chi(M) \cb(\mp N) \> ds_1 \> ds_2.
\end{align*}
We write the inner sum over $M$ and $N$ as
\begin{align*}
\chi(\mp 1) \left( L^4(1/2+s_1, \chi)L^4(1/2+s_2, \cb) \lambda(g, s_1, s_2, \chi) \theta (g, s_1, s_2, \chi ; d)- \tau_4^2(g)\right),
\end{align*}
where $\lambda(g, s_1, s_2, \chi)$ is holomorphic when $\tRe(s_1), \tRe(s_2) > \epsilon$. If $\tRe(s_1), \tRe(s_2) = \frac{100}{\log Q},$ $\lambda(g, s_1, s_2, \chi)$ is 
\begin{align*}
&\prod_p\left(1+\sum_{k, l\geq 1} \frac{\chi(p^k)\tau_4(p^k)}{p^{k(1/2+s_1)}}\frac{\overline{\chi(p^l)}\tau_4(p^l)}{p^{l(1/2+s_2)}}\left(1+ \sum_{k\geq 1} \frac{\chi(p^k)\tau_4(p^k)}{p^{k(1/2+s_1)}} + \sum_{k\geq 1} \frac{\overline{\chi(p^k)}\tau_4(p^k)}{p^{k(1/2+s_2)}}\right)^{-1} \right)^{-1}\\
&\ll (\log Q)^{16}.
\end{align*}

Also, for these values of $s_1, s_2,$ $\theta(g, s_1, s_2, \chi; d)$ is 
\begin{align*}
& \tau_4(g) \prod_{p|dg} \left( 1+\sum_{k\geq 1} \frac{\chi(p^k)\tau_4(p^k)}{p^{k(1/2+s_1)}} + \sum_{k\geq 1} \frac{\overline{\chi(p^k)}\tau_4(p^k)}{p^{k(1/2+s_2)}}\right)^{-1}  \prod_{p^r||g} \left( \sum_{k\geq 0} \frac{\tau_4(p^{r+k})}{p^{k(1/2+s_1)}}+ \sum_{k\geq 1} \frac{\tau_4(p^{r + k})}{p^{k(1/2+s_2)}} \right) \\
&\ll \tau^3 (dg)\tau_4^2(g),
\end{align*}
where $\tau(n)$ is the usual divisor function. Since $\chi$ is not principal, we can shift contours to $\tRe (s_1) = \tRe(s_2) = \frac{100}{\log Q}$ without passing any poles.  There, $\mathcal{EG}_1(\Psi, Q)$ is bounded by
\begin{align}
\label{eqn:offdiag}
&\ll Q (\log Q)^{16} \sum_{\substack{a \leq 2Q \\ b, h > 0}} \sum_{\substack{\chi \ ({\rm mod} \ abh) \\ \chi \neq \chi_0}} \sum_{\substack{g \\ b|g, (a,g) = 1}} \sum_{\substack{d \leq D \\ (d, g) = 1}}  \frac{\tau^3(d)\tau^3(g)\tau_4(g) }{ad g\phi(abh)} \notag \\
& \times \int_{(\frac{100}{\log Q})}\int_{(\frac{100}{\log Q})} \left(\left|L^4(1/2+s_1, \chi)L^4(1/2+s_2, \overline{\chi})\right| + \tau_4^2 (g) \right) \left|\Wt_2^\pm\left(s_1, s_2; \frac{Qd}{gh \lotwo}\right)\right| ds_1 ds_2.
\end{align} 
From Lemma \ref{lem:MellinXY}, we have that for any $k \geq 1$ and $t = 1$ or 3,
\begin{align}
\label{eqn:offdiagdivsum}
\begin{split}
\sum_{\substack{g \\ b|g}} & \frac{\tau^3(g) \tau_4^t(g)}{g} \sum_{d \leq D}\frac{\tau^3(d)}{d}  \left|\widetilde{W}_2\left(s_1, s_2, \frac{Qd}{gh\lotwo}\right)\right|  \\
&\ll  \frac{\left(1 + \frac{QD}{bh\lotwo}\right)^{k-1}}{\max(|s_1|, |s_2|)^k} (\log D)^6\sum_{\substack{g \\ b|g}} \frac{\tau^3(g) \tau_4^t(g)}{g} \exp\left(-c \left( \frac{gh\lotwo}{QD}\right)^{1/4}\right) \\
&\ll \frac{\left(1 + \frac{QD}{bh\lotwo}\right)^{k-1}}{\max(|s_1|, |s_2|)^k} \frac{(\log Q)^{\beta} \tau^3(b) \tau_4^3(b) }{b}\exp\left(-c \left( \frac{bh\lotwo}{QD}\right)^{1/4}\right).
\end{split}
\end{align}

Now examine dyadic intervals $S_1\leq |s_1| < 2S_1$, $S_2\leq |s_2| < 2S_2$, $A \leq a < 2A,$ $B \leq b < 2B,$ and $H \leq h < 2H.$ We write $\ell = abh.$  Note that $\phi(abh) \gg abh \log\log (abh)$, and $\tau^3(b) \tau_4^3(b) \leq \tau^3(\ell)\tau_4^3(\ell)$.  Thus, the contribution to such a dyadic block is $Q (\log Q)^{\beta}$ times a quantity
\begin{align}
\label{eqn:offdiagdiadic}
\begin{split}
&\ll  \frac{\left(1 + \frac{QD}{BH\lotwo}\right)^{k-1} \log^\epsilon (ABH)}{A^2B^2H\max(S_1, S_2)^k} \exp\left(-c \left( \frac{BH\lotwo}{QD}\right)^{1/4}\right) \sum_{ABH \leq \ell < 8ABH} \tau^6(\ell) \tau_4^3(\ell) \\
&\times \sum_{\substack{\chi \ ( {\rm mod} \ \ell) \\ \chi \neq \chi_0}} \int_{S_1 \leq |s_1| < 2S_1} \int_{S_2 \leq |s_s| < 2S_2} \left( 1 + \left|L^4(1/2+s_1, \chi)L^4(1/2+s_2, \overline{\chi})\right|\right)  \>ds_1\>ds_2.
\end{split}
\end{align}
Let $S = \max(S_1, S_2)$.  By Proposition \ref{thm:8momentIntandSumoverq},
\begin{align}
\label{eqn:offdiag8m}
\begin{split}
&\sum_{\substack{\chi \ ( {\rm mod} \ \ell) \\ \chi \neq \chi_0}} \int_{S_1 \leq |s_1| < 2S_1} \int_{S_2 \leq |s_s| < 2S_2} \left( 1 + \left|L^4(1/2+s_1, \chi)L^4(1/2+s_2, \overline{\chi})\right|\right)  \>ds_1\>ds_2 \\
&\leq \sum_{\substack{\chi \ ( {\rm mod} \ \ell) \\ \chi \neq \chi_0}} \int_{S_1 \leq |s_1| < 2S_1} \int_{S_2 \leq |s_s| < 2S_2} \left( 1 + \left|L(1/2+s_1, \chi)\right|^8+\left|L(1/2+s_2, \overline{\chi})\right|^8\right)  \>ds_1\>ds_2 \\
&\ll lS_1S_2 \log^{16+\epsilon}(l(S+1)).
\end{split}
\end{align}

We let $\sumd$ denote a dyadic sum. By (\ref{eqn:offdiagdivsum}), (\ref{eqn:offdiagdiadic}), and (\ref{eqn:offdiag8m}) above, (\ref{eqn:offdiag}) is bounded by $Q\log^\beta Q$ times

\begin{align*}
& \sumd_{A, B, H, S} \frac{\left(1 + \frac{QD}{BH\lotwo}\right)^{k-1} \log^{16+\epsilon} (ABH(S+1))}{ABS^{k-2}} \exp\left(-c \left( \frac{BH\lotwo}{QD}\right)^{1/4}\right) \times\\
& \hskip 2.5in \times \sum_{ABH \leq \ell < 8ABH} \tau^6(\ell) \tau_4^3(\ell)\\
&\ll \sumd_{B, H, S} \frac{H\left(1 + \frac{QD}{BH\lotwo}\right)^{k-1} \log^{\beta} ((S+1)QBH)}{S^{k-2}} \exp\left(-c \left( \frac{BH\lotwo}{QD}\right)^{1/4}\right),
\end{align*}where we have removed the sum over $A$ by using that $A\leq 2Q$.
We take $k = 1$ if $S \leq 1 + \frac{QD}{BH \lotwo}$ and $k =4$ otherwise. The contribution from $S \leq 1 + \frac{QD}{BH \lotwo}$ is 
\begin{align*}
&\ll \sumd_{B, H } H\left(1 + \frac{QD}{BH\lotwo}\right) \log^{\beta} (QBH) \exp\left(-c \left( \frac{BH\lotwo}{QD}\right)^{1/4}\right) \\
&\ll \sumd_{L} L \left(1 + \frac{QD}{L\lotwo}\right)\log^{\beta} (QL) \exp\left(-c \left( \frac{L\lotwo}{QD}\right)^{1/4}\right) \\
&\ll \frac{QD}{(\log Q)^{2\alpha - \beta}}.
\end{align*}
When $S > 1 + \frac{QD}{BH \lotwo}$, picking $k=4$ and excuting the dyadic sum over $S$ gives that the contribution is also bounded by
\begin{align*}
\sumd_{B, H } H\left(1 + \frac{QD}{BH\lotwo}\right) \log^{\beta} (QBH) \exp\left(-c \left( \frac{BH\lotwo}{QD}\right)^{1/4}\right) \ll \frac{QD}{(\log Q)^{2\alpha - \beta}}.
\end{align*}
To summarize our work, we have obtained that $ \mathcal{EG}_1(\Phi, Q) \ll \frac{Q^2D}{(\log Q)^{2\alpha - \beta}}.$  If $D = (\log Q)^\delta$ for a fixed number $\delta$ independent of $\alpha$, then 
$$ \mathcal{EG}_1(\Phi, Q) \ll \frac{Q^2}{\log Q},$$
provided we pick $\alpha$ large enough. 
 
\end{proof}

From (\ref{eqn:EG2}) and Lemma \ref{lem:EG1bound}, assuming GRH, we obtain that
\begin{align} \label{eqn:summaryEG}
 \mathcal{EG}(\Phi, Q) \ll \frac{Q^2}{\log Q}.
\end{align}

\section{Evaluating $\mathcal{MS}(\Psi, Q) + \Mg$ } \label{sec:MsplusMg}
We recall that $\mathcal{MS}(\Psi, Q)$ and $\Mg$ are defined in
(\ref{eqn:mainMS}) and  (\ref{eqn:maintermG}) respectively. To evaluate $\mathcal{MS}(\Psi, Q) + \Mg$, we use the Mellin transform of $\Wt_1^\pm$ in Lemma \ref{lem:MellinU} to write $\mathcal{MG}(\Psi, Q)$ in term of an integral on the vertical line with real part $-\epsilon < 0$.  Shifting the contour to $\R(z) = \epsilon > 0$, we pick up a pole at $z = 0.$  We will show that the contribution of this pole cancels with $\mathcal{MS}(\Psi, Q)$.

\begin{lem} \label{lem:MSplusMG} We have
\begin{align} \label{eqn:maintermsumMSandMG} 
\begin{split}
\mathcal{MS}(\Phi, Q) + \Mg = \frac{Q}{2} &\sum_{\substack{m, n = 1}}^{\infty} \frac{\tau_4 (m) \tau_4 (n)}{\sqrt{mn}}  \frac{1}{2\pi i} \int_{(\epsilon)} \Wt^\pm_1\left(\frac{m \lotwo}{Q^2}, \frac{n \lotwo}{Q^2} ; z\right) \\
&\times \frac{\zeta(1 - z) \F(-z, g, MN)}{\zeta(1 + z) \phi(gMN, 1+ z)} \left(\frac{Q}{g \lotwo}\right)^{-z} \> dz + O\left( \frac{Q^2}{\log Q}\right).
\end{split}
\end{align}
\end{lem} 

\begin{proof}
We first simpify the sum over $a, b$ in $\Mg$ by letting $r = ab$ so that
$$\mathcal{A}(h, g, MN) = \sum_{(a,gMN) = 1} \sum_{\substack{b|g \\ (b, MN)}}\frac{\mu(a)\mu(b)}{a\phi(abh)} = \sum_{(r, MN) = 1} \frac{\mu(r) (r, g)}{r\phi(rh)}.$$
Next we will evaluate sum over $h,$ which is given by
\begin{align*} 
\sum_{(h, MN) = 1} \mathcal{A}(h, g, MN)   \W^\pm \left(\frac{gM \lotwo}{Q^2}, \frac{gN \lotwo}{Q^2} ; \frac{Qd}{gh \lotwo}\right).
\end{align*}
Using the Mellin transform of $\Wt_1^\pm$ given in Lemma \ref{lem:MellinU} with $c = - \epsilon < 0$, we obtain that the above sum is
\begin{equation} \label{eqn:maintermGsumoverh}
\sum_{\substack{h = 1 \\(h, MN) = 1}}^{\infty} \mathcal{A}(h, g, MN)   \frac{1}{2\pi i} \int_{(-\epsilon)} \Wt_1^\pm \left(\frac{gM \lotwo}{Q^2}, \frac{gN \lotwo}{Q^2} ; z\right) \left(\frac{Qd}{gh \lotwo}\right)^{-z} \> dz.
\end{equation}
We can interchange the sum and the integral since the sum over $h$ is absolutely convergent for $\Re (z)< 0$. Writing out the Euler product, we obtain that
$$ \sum_{\substack{h = 1 \\(h, MN) = 1}}^{\infty} \frac{\mathcal{A}(h, g;MN)}{h^s} = \zeta(s+1) \F(s, g, MN),$$
where 
\begin{equation} \label{eqn:defofcalF}
\F(s, g, MN) = \phi(MN, s + 1) \prod_{p \nmid gmN} \left(1 - \frac{1}{p(p-1)} + \frac{1}{p^{1 + s}(p-1)}\right) \prod_{\substack{p|g \\ p \nmid MN}} \left( 1 - \frac{1}{p^{1 + s}} - \frac{1}{p-1} \left(1 - \frac{1}{p^s}\right)\right),
\end{equation} 
and 
$$ \phi(r, s)= \prod_{p | r} \left( 1 - \frac{1}{p^s}\right).$$
Therefore (\ref{eqn:maintermGsumoverh}) is
\begin{align}  \label{eqn:sumoverHsimplify}
&\frac{1}{2\pi i} \int_{(-\epsilon)} \Wt^\pm_1\left(\frac{gM \lotwo}{Q^2}, \frac{gN \lotwo}{Q^2} ; z\right) \zeta(1 - z) \F(-z, g, MN) \left(\frac{Qd}{g \lotwo}\right)^{-z} \> dz \\
&= -({\rm Residue \ at} \ z = 0) + \nonumber\\
& \ \ + \frac{1}{2\pi i} \int_{(\epsilon)} \Wt^\pm_1\left(\frac{gM \lotwo}{Q^2}, \frac{gN \lotwo}{Q^2} ; z\right) \zeta(1 - z) \F(-z, g, MN) \left(\frac{Qd}{g \lotwo}\right)^{-z} \> dz. \nonumber
\end{align}
By the definition of $\mathcal F(0, g, MN)$, we have that $-({\rm Residue \ at} \ z = 0)$ is  
\begin{align*} 
&\Wt^\pm_1\left(\frac{gM \lotwo}{Q^2}, \frac{gN \lotwo}{Q^2} ; 0\right) \zeta(1 - z) \F(0, g, MN) \nonumber \\
&= \frac{\phi(mn)}{mn} \int_0^\infty  \Psi(u) V\left(\frac{m \lotwo}{Q^2}, \frac{n \lotwo}{Q^2} ; u\right) du.
\end{align*}
Therefore the contribution to the residue term at $z = 0$ is
\begin{equation} \label{eqn:residue0ofMg}
 Q \sum_{\substack{m, n = 1  \\ m \neq n}} \frac{\tau_4 (m) \tau_4 (n)}{\sqrt{mn}} \sum_{\substack{d \leq D \\ (d, mn) = 1}}\frac{\mu(d)}{d}\frac{\phi(mn)}{mn} \int_0^\infty  \Psi(u) V\left(\frac{m \lotwo}{Q^2}, \frac{n \lotwo}{Q^2} ; u\right) du.
\end{equation}

We will show that the sum 
\begin{equation} \label{eqn:MSMgcancel}
 \textrm{Equation (\ref{eqn:residue0ofMg})} \  + \mathcal{MS}(\Psi, Q) = O\left(\frac{DQ^2(\log Q)^{\beta}}{(\log Q)^{\alpha/2}}\right).
\end{equation}  
The bound above is an acceptable error term when we choose appropriate values for $D$ and $\alpha.$ We recall the definition of $\mathcal {MS}(\Psi, Q)$ is
$$
- \sum_{\substack{m, n = 1 \\ m \neq n}} \frac{\tau_4 (m) \tau_4 (n)}{\sqrt{mn}} \sum_{\substack{d \leq D \\ (d, mn) = 1}} \mu(d) \sum_{\substack{(r, mn) = 1}}\Psi \left( \frac{dr}{Q}\right)  V\left(m, n ; \frac{dr}{\lon} \right).$$
By elementary arguments, we have
$$ \sum_{\substack{r \leq x \\ (r, mn) = 1}} 1 = \frac{\phi(mn)}{mn} x + O(\tau(mn)), $$
where $\tau$ is the divisor function. By partial summation, the main contribution from $\mathcal {MS}(\Psi, Q)$ is 
\begin{align*}
&-Q \sum_{\substack{m, n = 1 \\ m \neq n}} \frac{\tau_4 (m) \tau_4 (n)}{\sqrt{mn}} \sum_{\substack{d \leq D \\ (d, mn) = 1}} \frac{\mu(d)}{d} \frac{\phi(mn)}{mn} \int_0^\infty \Psi(u) V\left(m, n ; \frac{uQ}{\lon} \right) \> du = -\textrm{ Equation (\ref{eqn:residue0ofMg})}
\end{align*}
by relation (\ref{eqn:Vc}). Since $V$ is small unless $\max (m,n) \leq \frac{Q^2}{(\log Q)^{\alpha/2}}$, we then obtain that $\mathcal{MS}(\Psi, Q)$ + Equation (\ref{eqn:residue0ofMg}) is bounded by $\frac{D Q^{2}(\log Q)^{\beta}}{(\log Q)^{\alpha/2}}$, where $\beta$ does not depend on $\alpha$, and this proves (\ref{eqn:MSMgcancel}).
\\

Therefore the main term of $\mathcal{MS}(\Psi, Q) + \Mg$ is
\begin{align*}
\frac{Q}{2} \sum_{\substack{m, n = 1 \\ m \neq n}} \frac{\tau_4 (m) \tau_4 (n)}{\sqrt{mn}}  \frac{1}{2\pi i} &\int_{(\epsilon)} \Wt^\pm_1\left(\frac{m \lotwo}{Q^2}, \frac{n \lotwo}{Q^2} ; z\right) \\
&\times \zeta(1 - z) \F(-z, g, MN) \left(\frac{Q}{g \lotwo}\right)^{-z}  \sum_{\substack{d \leq D \\ (d, gMN) = 1}} \frac{\mu(d)}{d^{1 + z}}\> dz . 
\end{align*}
Now we shift the contour to ${\rm Re}(z) = 1 - \frac{1}{\log Q}.$ We extend the sum over $d$ to all positive integers. By Lemma \ref{lem:MellinU}, $\Wt^\pm_1$ is small unless $m, n$ and $ g \leq Q^2$. By Equation (\ref{eqn:defofcalF}), the error term in extending is 
\begin{align} \label{error:extendD}
&\ll Q  \sum_{\substack{m, n = 1 \\ m \neq n}} \frac{\tau_4 (m)\tau(M) \tau_4 (n) \tau(N) \tau^2(g)}{\sqrt{mn}} \int_{\left(1 - \tfrac{1}{\log Q}\right)} \left| \frac{(m \pm n)\lotwo}{Q^2}\right|^{-1 +1/\log Q} |z|^{-10}  \nonumber \\
& \hskip 1.5in \times\exp \left(-c \left( \frac{\max (m,n) \lotwo}{Q^2}\right)^{1/4} \right)\left(\frac{Q}{g \lotwo}\right)^{-1 +1/\log Q}  \frac{1}{D} \ |dz| \nonumber \\
&\ll \frac{Q^2}{D}   \sum_{g \leq Q^2} \frac{\tau_4^2(g)\tau^2(g)}{g} \sum_{\substack{m, n \leq Q^2 \\ m \neq n}} \frac{\tau_4 (m) \tau(m) \tau_4(n) \tau(n)}{\sqrt{mn}} \frac{1}{|m \pm n|}.
\end{align}
Since $m + n \geq 2\sqrt{mn},$ we have
$$ \sum_{\substack{m, n \leq Q^2 \\ m \neq n}} \frac{\tau_4 (m) \tau(m) \tau_4(n) \tau(n)}{\sqrt{mn}} \frac{1}{|m + n|} \ll  \sum_{\substack{m, n \leq Q^2 }} \frac{\tau_4 (m) \tau(m) \tau_4(n) \tau(n)}{mn} \ll (\log Q)^{16}. $$ 
For the other term, we may assume by symmetry that $m<n$ and applying Cauchy-Schwarz inequality, we have
\begin{align*}
\sum_{\substack{m, n \leq Q^2 \\ m < n}} \frac{\tau_4 (m) \tau(m) \tau_4(n) \tau(n)}{\sqrt{mn} |m - n|} &\ll \sum_{\substack{m \leq Q^2}} \frac{\tau_4 (m) \tau(m) }{\sqrt{m}}   \sum_{ \ell \leq Q^2 }   \frac{\tau_4(m + \ell) \tau(m + \ell)}{\ell \sqrt{m + \ell} } \\
&\ll \sum_{\substack{\ell \leq Q^2}} \frac{1}{\ell} \left(\sum_{ m \leq Q^2 }  \frac{\tau_4^2(m) \tau^2(m) }{m}  \right)^{1/2}  \left( \sum_{ m \leq Q^2 } \frac{\tau_4^2(m + \ell) \tau^2(m + \ell)}{m + \ell }\right)^{1/2}  \\
&\ll (\log Q)^{65}.
\end{align*}
Therefore (\ref{error:extendD}) is bounded above by $\frac{Q^2(\log Q)^{129}}{D}.$

We now move the contour back to ${\rm Re} (z) = \epsilon$ and reinsert the terms $m = n$ with a negligible error of $O(Q^{1 + \epsilon})$. Hence up to the error term of $ \frac{D Q^{2}(\log Q)^{\beta}}{(\log Q)^{\alpha/2}} + \frac{Q^2(\log Q)^{129}}{D}$, \ $\mathcal{MS}(\Psi, Q) + \Mg$ is 
\begin{align*} 
\frac{Q}{2} \sum_{\substack{m, n = 1}}^{\infty} \frac{\tau_4 (m) \tau_4 (n)}{\sqrt{mn}}  \frac{1}{2\pi i} \int_{(\epsilon)} &\Wt^\pm_1\left(\frac{m \lotwo}{Q^2}, \frac{n \lotwo}{Q^2} ; z\right) \\
&\times \frac{\zeta(1 - z) \F(-z, g, MN)}{\zeta(1 + z) \phi(gMN, 1+ z)} \left(\frac{Q}{g \lotwo}\right)^{-z} \> dz. \nonumber
\end{align*}
Choosing $D = (\log Q)^{130}$ and $\alpha$ large enough, we obtain that the error term is $O\left(\frac{Q^2}{\log Q}\right).$
\end{proof}
\section{The off-diagonal contribution} \label{sec:offdiagmain}
In this section, we will evaluate the main term in  (\ref{lem:MSplusMG}) as the following Proposition. 

\begin{prop} \label{prop:mainMG} Assume the Lindel\"{o}f hypothesis. \footnote{The Lindel\"{o}f hypothesis states that $|\zeta(1/2 + it)| \ll_{\epsilon} (|t| + 1)^{\epsilon}$ and is a well known consequence of the Riemann Hypothesis.} We have that
$$  \mathcal{MS}(\Psi, Q) + \Mg = \frac{-53524}{16!}Q^2(\log Q)^{16}\frac{\psit(2)}{2} \frac{\mathcal{K}\left(\h, \h; 1 \right)}{\zeta(2)}\int_{-\infty}^{\infty} G\left(\h, t \right) \> dt + O(Q^2 (\log Q)^{15 }).$$
\end{prop}
\begin{proof}
First we use a Mellin transform in Lemma \ref{lem:MellinXYU} and obtain that (\ref{eqn:maintermsumMSandMG}) is
\begin{align} \label{eqn:sumMSandMGafterMellin3}
\frac{Q}{2} \frac{1}{(2\pi i)^3} \int_{(\epsilon)}  \int_{(1/2 + \epsilon)} \int_{(1/2 + \epsilon)}& \Wt_3(s_1, s_2 ; z) \frac{\zeta(1 - z)}{\zeta(1 + z)} \frac{Q^{2s_1 + 2s_2 - z}}{(\log Q)^{2\alpha(s_1 + s_2 - z)}}  \mathcal{J}(s_1, s_2; z) \>ds_2 \>ds_1 \> dz, 
\end{align}
where 
\begin{align*}
\mathcal{J}(s_1, s_2; z) = \sum_{\substack{m, n = 1}}^{\infty} \frac{\tau_4 (m) \tau_4 (n)}{m^{1/2 + s_1} n^{1/2 + s_2}} \frac{g^z \F(-z, g, MN)}{\phi(gMN, 1+ z)}.
\end{align*}
The coefficients above are multiplicative, so we can write $\mathcal{J}(s_1, s_2; z) = \prod_p \mathcal{J}_p(s_1, s_2; z), $
where
\begin{align} \label{eqn:prodofB}
\mathcal{J}_p(s_1, s_2; z) = &1 + \frac{p^z - 1}{p(p-1)} + \sum_{\substack{a, b \geq 0 \\ \max(a,b) \geq 1}} \frac{\tau_4(p^a)\tau_4(p^b)}{p^{a(1/2 + s_1)}p^{b(1/2 + s_2)}} p^{z\min(a,b)} \frac{1 - \tfrac{1}{p^{1 -z}}}{1 - \tfrac{1}{p^{1 + z}}} \\
&+ \frac{1}{(p-1)} \frac{p^z - 1}{ 1 - \tfrac{1}{p^{1 + z}}} \sum_{ k = 1}^\infty \frac{\tau_4^2(p^k)}{p^{k(1 + s_1 + s_2 - z)}}. \nonumber
\end{align}
Thus $\mathcal{J}(s_1, s_2; z)$ as
\begin{align*}
\zeta(2-z) \frac{\zeta^4\left(\frac{1}{2} + s_1\right)\zeta^4\left(\frac{1}{2} + s_2\right)}{\zeta^4\left(\frac{3}{2} + s_1 - z\right)\zeta^4\left(\frac{3}{2} + s_2 - z\right)} \zeta^{16}(1 + s_1 + s_2 - z) \mathcal{K} (s_1, s_2; z),
\end{align*}
where $\mathcal{K} (s_1, s_2; z) = \prod_p \mathcal{K}_p(s_1, s_2; z)$ is absolute convergent in $\tRe (s_1) > 0, \tRe (s_2) > 0$ and $\tRe (s_1 + s_2) > \tRe(z) - 1/2. $ 

Before moving the line of integration, we first truncate the integrals in $s_1$ and $s_2$ at a height $T = Q^2.$ Using the bound for $\Wt_3$ from Equation (\ref{eqn:boundWt3}), we obtain that the error term from the truncation is $O(Q^{1 + \epsilon}).$  Now we move the line of integration in $s_1$ and $s_2$ to $\tRe(s_1) = \tRe(s_2) = \frac{2}{\log Q}$, encountering poles of order four at $s_1 = 1/2$ and $s_2 = 1/2.$   Using the Lindel\"{o}f hypothesis and the bound (\ref{eqn:boundWt3}) for $\Wt_3$, the remaining integral is also $O(Q^{1 + \epsilon})$, so we need only study the contribution from the residues. 

Recall that 
$$ \Wt_3(s_1, s_2; z) = \frac{\psit(1 + 2s_1 + 2s_2 - z)}{(s_1 + s_2 - z) \pi^{2s_1 + 2s_2 - 2z}} \int_{-\infty}^{\infty} \Hc \left(\frac{s_1 - s_2 + z}{2} - it, z\right) G\left(\h + \frac{s_1 + s_2 - z}{2}, t \right) \> dt.$$

Thus the residue from $s_1 = s_2= 1/2$ is 
\begin{align} \label{eqn:resatonehalf}
&\frac{Q}{2} \frac{1}{2\pi i} \int_{(\epsilon)} \frac{\psit(3 - z)}{ \pi^{2 - 2z}}\left( \int_{-\infty}^{\infty} \Hc \left(\frac{z}{2} - it, z\right) G\left(\h + \frac{1 - z}{2}, t \right) \> dt\right) \frac{\zeta(1 - z)}{\zeta(1 + z)} \frac{\zeta(2-z) \mathcal{K}\left(\frac{1}{2}, \frac{1}{2}; z\right)}{\zeta^8(2-z)} \\
& \times  \left({\rm Res}_{s_1 = s_2 = 1/2} \frac{\zeta^4\left(\frac{1}{2} + s_1\right)\zeta^4\left(\frac{1}{2} + s_2\right) \zeta^{16}(1 + s_1 + s_2 - z)}{(s_1 + s_2 -z)} \frac{Q^{2s_1 + 2s_2 - z}}{(\log Q)^{2\alpha(s_1 + s_2 - z)}}\right) \> dz. \nonumber
\end{align}Note this contributes a factor of $\log^6 Q$ to the leading term. From the definition of $\Hc$ in Lemma \ref{lem:MellinXYU}, we have
$$ \Hc \left(\frac{z}{2} - it, z\right) = \pi^{1/2} \frac{\Gamma\left(\tfrac{z}{4} - \tfrac{it}{2} \right)\Gamma\left(\tfrac{1-z}{2} \right)\Gamma\left(\tfrac{z}{4} + \tfrac{it}{2} \right)}{\Gamma\left(\tfrac{1}{2} - \tfrac{z}{4} + \tfrac{it}{2} \right)\Gamma\left(\tfrac{z}{2} \right)\Gamma\left(\tfrac{1}{2} - \tfrac{z}{4} - \tfrac{it}{2}\right)}.$$
We now move the integration in $z$ to $\tRe(z) = \frac{3}{2} - \frac{1}{\log Q}$. The remaining integral gives an acceptable error term, and we encounter a pole of order $11$ at $z = 1$ (there is a simple pole from $\Hc \left(\frac{z}{2} - it, z\right)$, a zero of order 7 from $\frac{1}{\zeta^7(2-z)}$ and a pole of order 17 from the expression ${\rm Res}_{s_1 = s_2 = 1/2}$). Therefore the leading term of the residue at $z = 1$ has order $Q^2(\log Q)^{16}.$  By Maple (or a tedious calculation by hand), we obtain that the leading term is
\begin{align*} 
& \frac{13381}{2615348736000}Q^2(\log Q)^{16} \zeta(0) \frac{\psit(2)}{2} \frac{\mathcal{K}\left(\h, \h; 1 \right)}{\zeta(2)}\int_{-\infty}^{\infty} G\left(\h, t \right) \> dt \\
&= \frac{-53524}{16!}Q^2(\log Q)^{16}\frac{\psit(2)}{2} \frac{\mathcal{K}\left(\h, \h; 1 \right)}{\zeta(2)}\int_{-\infty}^{\infty} G\left(\h, t \right) \> dt.
\end{align*}
This concludes the proof of the proposition. 

\end{proof}
\section{Conclusion of the proof of Theorem \ref{thm:eightmoment}}
The proof of Theorem \ref{thm:eightmoment} will follow from the lemma below.

\begin{lem} Assume GRH. We have
\begin{align*}  
&\D + \mathcal{S}(\Psi, Q) + \mathcal{G}(\Psi, Q) \\
&= \frac{12012}{16!} Q^2(\log Q)^{16}\frac{\psit(s)}{2} A(1/2) \\
&  \hskip 0.3in \times \prod_p\left( 1 - \frac{1}{p}\right)\left(1 + \frac{1}{B_p(1/2)} \left( \frac{1}{p} - \frac{1}{p^2} - \frac{1}{p^3}\right)\right) \int_{-\infty}^{\infty} G(1/2, t)dt + O(Q(\log Q)^{15}).
\end{align*}
\end{lem}
\begin{proof}
From Proposition (\ref{prop:sumSm}), Equation (\ref{eqn:summaryEG}), and Proposition (\ref{prop:mainMG}), we obtain that
$$ \mathcal{S}(\Psi, Q) + \mathcal{G}(\Psi, Q) = \frac{-53524}{16!}Q^2(\log Q)^{16}\frac{\psit(2)}{2} \frac{\mathcal{K}\left(\h, \h; 1 \right)}{\zeta(2)}\int_{-\infty}^{\infty} G\left(\h, t \right) \> dt + O(Q^2 (\log Q)^{15 + \epsilon}).$$
Moreover, by Equation (\ref{eqn:eulerfordiagonal}) in Proposition \ref{sec:diagonal}, we have that $\D$, up to error term of $O(Q^2 (\log Q)^{15 + \epsilon})$, is
\begin{align*}
 2^{16}Q^2 \frac{(\log Q )^{16}}{16!} \frac{\widetilde{\Psi}(2)}{2} A(1/2)\prod_p\left( 1 - \frac{1}{p}\right)\left(1 + \frac{1}{B_p(1/2)} \left( \frac{1}{p} - \frac{1}{p^2} - \frac{1}{p^3}\right)\right) \int_{-\infty}^{\infty} G(1/2, t)dt.
\end{align*}
To derive an expression for $\mathcal{D}(\Psi, Q) + \mathcal{S}(\Psi, Q) + \mathcal{G}(\Psi, Q)$, we first show that 
$$ \frac{\mathcal{K}\left(\h, \h; 1 \right)}{\zeta(2)} = A(1/2)\prod_p\left( 1 - \frac{1}{p}\right)\left(1 + \frac{1}{B_p(1/2)} \left( \frac{1}{p} - \frac{1}{p^2} - \frac{1}{p^3}\right)\right).$$
It suffices to check that
\begin{equation} \label{eqn:checkEulerp}
\mathcal{K}_p\left(\h, \h; 1 \right)\left( 1 - \frac{1}{p^2}\right) = A_p(1/2)\left( 1 - \frac{1}{p}\right)\left(1 + \frac{1}{B_p(1/2)} \left( \frac{1}{p} - \frac{1}{p^2} - \frac{1}{p^3}\right)\right).
\end{equation}
By the definition of $A_p$, $B_p$, and using Lemma \ref{lem:eulerprod} ,$$\mathcal{K}_p\left(\h, \h; 1 \right) = \mathcal{J}_p\left(\h, \h; 1 \right)\left(1 - \frac{1}{p}\right)\left( 1 - \frac{1}{p}\right)^{16},$$
and
$$ \mathcal{J}_p\left(\h, \h; 1 \right) = 1 + \frac{1}{p} + \left( 1 - \frac{1}{p^2}\right)^{-1} \sum_{k = 1}^{\infty} \frac{\tau_4^2(p^k)}{p^k},$$
from which (\ref{eqn:checkEulerp}) follows. We then obtain the lemma.

\end{proof}
Hence, up to error term of $O(Q^2(\log Q)^{15 + \epsilon})$

\begin{align*}
& \sum_{q}\Psi \left( \frac{q}{Q}\right) \sumb_{\chiq } \int_{-\infty}^{\infty} \left| \Lambda\left(\tfrac{1}{2} + iy, \chi \right)\right|^8 \> dy = 2(\D + \mathcal{S}(\Psi, Q) + \mathcal{G}(\Psi, Q) ) \\
&= 2 \cdot \frac{12012}{16!} Q^2(\log Q)^{16}\frac{\psit(s)}{2} A(1/2) \\
& \hskip 1in \times \prod_p\left( 1 - \frac{1}{p}\right)\left(1 + \frac{1}{B_p(1/2)} \left( \frac{1}{p} - \frac{1}{p^2} - \frac{1}{p^3}\right)\right) \int_{-\infty}^{\infty} G(1/2, t)dt \\
&= \frac{24024}{16!} a_4 \sum_{q}\Psi \left( \frac{q}{Q}\right) \phib (\log q)^{16} \prod_{p | q} B_p(1/2) \int_{-\infty}^{\infty} G(1/2, t)dt. \\
\end{align*}
This concludes the proof of Theorem \ref{thm:eightmoment}.

\section{The fourth moment of Dirichlet twists of a $GL(2)$ automorphic $L$-functions } \label{sec:twistedfourthmoment}
In this section, we will discuss how to modify the proof of Theorem \ref{thm:eightmoment} to prove Theorem \ref{thm:twisted4moment}. The proof may be carried in the same manner, and indeed is slightly easier as the main term contribution comes solely from the diagonal term. 

Define 
$$G_f(s, t) := \Gamma^2\left(\frac{k - 1}{2} + s + it \right)\Gamma^2\left(\frac{k - 1}{2} + s- it\right); \ \ \ \ \ \ \sigma_f(n) = \sum_{n = n_1n_2} a(n_1) a(n_2);$$
$$W_f(x, t) := \frac{1}{2\pi i} \int_{(1)} G_f(1/2+s, t) x^{-s} \frac{ds}{s};
$$
$$ V_f(\xi, \eta; \mu) = \int_{-\infty}^{\infty} \left(\frac{\eta}{\xi}\right)^{it} W_f\left( \frac{\xi\eta (2\pi)^4}{\mu^4}, t\right) \> dt;$$
and 
$$ \Lambda_1(f \times \chi) = \sum_{m, n = 1}^{\infty} \frac{\sigma_f(m) \sigma_f(n)}{\sqrt{mn}} \chi(m) \overline{\chi}(n) V_f(m, n ;q).$$

By using the functional equation in (\ref{eqn:fncEqnoftwistedL}) and following the same arguments as Lemma \ref{prop:fnceqnLambda} , we can write
$$  \M_f = \sum_{q } \Psi\left(\frac{q}{Q}\right) \sumstar_{\chiq} \int_{-\infty}^{\infty} \left| \Lambda\big( \frac{1}{2} + iy, f\times\chi \big)\right|^4 \> dy = 2\Delta_f(\Psi, Q),$$
where
$$ \Delta_f(\Psi, Q) = \sum_{q} \sumstar_{\chiq} \Psi \left(\frac{q}{Q}\right) \Lambda_1(f \times \chi).$$
By orthogonality relation of $\sumstar$ in Lemma \ref{lem:orthogonal}, we obtain that
$$\Delta_f(\Psi, Q) =  \sum_{m, n=1}^\infty \frac{\sigma_f(m) \sigma_f(n)}{\sqrt{mn}}\sum_{\substack{d, r\\(dr, mn) = 1 \\ r|m- n}} \phi(r) \mu(d) \Psi\bfrac{dr}{Q} V_f(m, n, dr).
$$
By the same arguments as Section \ref{sec:truncation}, we can truncate the sum with acceptable error term by invoking Lemma \ref{prop:largesieve}. Therefore we consider
$$\dt_f\left(\Psi, Q\right) =  \sum_{m, n=1}^\infty \frac{\sigma_f(m) \sigma_f(n)}{\sqrt{mn}}\sum_{\substack{d, r\\(dr, mn) = 1 \\ r|m - n}} \phi(r) \mu(d) \Psi\bfrac{dr}{Q} V_f\left(m, n, \frac{dr}{\lon}\right),$$
We can write
\begin{equation*} 
\dt_f\left(\Psi, Q\right) = \Df  +\Smf + \Gf
\end{equation*}
where the diagonal term $\Df$ consists of the terms with $m=n$, the term $\Smf$ consists of the remaining terms with $d>D$ and $\Gf$ consists of the rest of the terms with $d \leq D$. 

Since $L(s, f \times \chi_0)$ is an entire function, we can treat $L(s, f \times \chi_0)$ in the same way as treating non-principle character twists. Therefore, we bound $\Smf$ and $\Gf$ in the same way as bounding non-principal character contributions of $\Sm$ (Section \ref{sec:calS}) and $\G$ (Section \ref{sec:errorEg} and using Proposition \ref{thm:4twistedmomentIntandSumoverq}). Moreoever, there is no main term contribution from off-diagonal terms. 

For the diagonal term $\Df,$ the argument is the same as in Section \ref{sec:diagonal}, but we use Lemma \ref{lem:eulerprodfortwistedfourth} instead. This completes the proof of Theorem \ref{thm:twisted4moment}.
\section{Acknowledgements}
The first author learned about the asymptotic large seive from Brian Conrey at the {\it Arithmetic Statistics} MRC program at Snowbird. She would like to thank both Brian Conrey and the organizers of the program. Part of this work was done while the second author was visiting Centre de Recherches Math\'{e}matiques - he is grateful for their kind hospitality.

We are grateful to Henryk Iwaniec, Micah Milinovich and Maksym Radziwill for useful comments on a draft of this paper.  We would also like to thank the anonymous referee for a number of helpful comments.


\begin{thebibliography}{9999}

\bibitem{CFKRS} J.B. Conrey, D. Farmer, J.Keating, M. Rubinstein, and N.Snaith, {\it Integral moments of L-functions,}
Proc. London Math. Soc. 91 (2005), 33-104.

\bibitem{CGh} J.B. Conrey and A. Ghosh, {\it A conjecture for the sixth power moment of the Riemann zeta-function.} Internat. Math. Res. Notices 1998, no. 15, 775 - 780.


\bibitem{CGo} J.B. Conrey and S.M. Gonek, {\it High moments of the Riemann zeta-function,} Duke Math. J. 107 (2001), no. 3, 577 - 604.

\bibitem{CIS} J.B. Conrey, H. Iwaniec and K. Soundararajan, {\it The sixth power moment of Dirichlet $L$-functions},  Geometric and Functional Analysis, 
October 2012, Volume 22, Issue 5, 1257-1288.

\bibitem{Da} H. Davenport, {\it Multiplicative Number Theory}, vol.74, Springer-Verlag (GTM), New York, 2000.

\bibitem{DGH} A. Diaconu, D. Goldfeld, and J. Hoffstein, {\it Multiple Dirichlet series and moments of zeta and $L$-functions.} Compositio Math. 139 (2003), 297-360.

\bibitem{IK} H. Iwaniec and E. Kowalski, {\it Analytic Number Theory}, vol. 53, American Mathematical Society Colloquium Publications, Rhode Island, 2004.

\bibitem{KaSa} N. Katz and P. Sarnak, {\it Random matrices, Frobenius eigenvalues, and monodromy.} American Mathematical Society Colloquium Publications, 45. American Mathematical Society, Providence, RI, 1999.

\bibitem{KS} J.P. Keating and N.C. Snaith, {\it Random matrix theory and $\zeta(\h + it),$} Comm.  Math. Phys. 214 (2000), 57-89. 

\bibitem{KS2} J.P. Keating and N.C. Snaith, {\it Random matix theory and $L$-functions at $s = 1/2$} Comm. Math. Phys. 214 (2000), 91 - 110.

\bibitem{MV} H.L. Montgomery and R.C. Vaughan, {\it Multiplicative number theory. I. Classical theory,} Cambridge University Press, Cambridge, 2007. 


\bibitem{Lower} Z. Rudnick and K. Soundararajan, {\it Lower bounds for moments of L-functions,} Proc. Natl. Acad. Sci. USA 102 (2005), 6837-6838.

\bibitem{Selberg} A. Selberg, {\it Contributions to the theory of the Riemann zeta functions}, Archiv Math. Naturvid. 48 (1946), 89-155. 

\bibitem{Selberg2} A. Selberg, {\it Old and new conjectures and results about a class of Dirichlet series}, Proceedings of the Amalfi Conference on Number Theory, Collected Papers (Vol. II), pp. 47-63.

\bibitem{Sound} K. Soundararajan, {\it Moments of the Riemann zeta-function}, Ann. of Math. (2) 170 (2009) 981 - 993.


\bibitem{SoundDiriMoment} K. Soundararajan, {\it The fourth moment of Dirichlet L-functions}, Analytic number theory, 239-246, Clay Math. Proc., 7, Amer. Math. Soc., Providence, RI, 2007.

\bibitem{SY} K. Soundararajan and M. Young, {\it The second moment of quadratic twists of modular L-functions}, J. Eur. Math. Soc. (JEMS) 12 (2010), no. 5, 1097-1116.

\bibitem{Te} G. Tenenbaum, {\it Introduction to analytic and probabilistic number theory}, Cambridge University Press, Cambridge, 1995.

\bibitem{Ti} E.C. Titchmarsh, {\it The Theory of the Riemann Zeta-function}, Oxford University Press, New York, 1986. 
\end{thebibliography}

\end{document}